\documentclass[11pt,letterpaper]{article}

\usepackage[utf8]{inputenc}
\usepackage{comment}
\usepackage{microtype}
\usepackage{graphicx}
\usepackage{booktabs} 
\usepackage{scalerel}
\usepackage[margin=1in]{geometry}
\usepackage[style=alphabetic, maxbibnames=15, maxcitenames=15, natbib=true,
  maxalphanames=10, backend=biber, sorting=nty, backref=true]{biblatex}
\DefineBibliographyStrings{english}{backrefpage = {page},backrefpages = {pages},}

\usepackage{graphicx}%
\usepackage{multirow}%
\usepackage{amsmath,amssymb,amsfonts}%

\usepackage{amsthm}%
\usepackage{mathrsfs}%
\usepackage[title]{appendix}%
\usepackage{xcolor}%
\usepackage{textcomp}%
\usepackage{manyfoot}%
\usepackage{booktabs}%
\usepackage{algorithm}%
\usepackage{comment}%
\usepackage{algorithmicx}%
\usepackage{algpseudocode}%
\usepackage{listings}%
\DeclareMathOperator*{\argmin}{arg\,min}
\newtheorem{assumption}{Assumption}
\newtheorem{theorem}{Theorem}
\newtheorem{proposition}{Proposition}
\newtheorem{lemma}{Lemma}
\newtheorem{remark}{Remark}
\newtheorem{corollary}{Corollary}

\newcommand{\bigO}{\mathcal{O}}
\newcommand{\reals}{\mathbb{R}}
\newcommand{\Tr}{\mathbf{Tr}}
\newcommand{\Det}{\mathbf{Det}}
\usepackage{makecell}
\usepackage[shortlabels]{enumitem}
\usepackage{subfigure}

\usepackage{hyperref}
\hypersetup{colorlinks,citecolor=blue,linktocpage,breaklinks=true}
\begin{document}

\title{Non-asymptotic Global Convergence Rates of BFGS\\ with Exact Line Search
}

\author{Qiujiang Jin\thanks{Department of Electrical and Computer Engineering, The University of Texas at Austin, Austin, TX, USA  \{qiujiang@austin.utexas.edu\}} \qquad Ruichen Jiang\thanks{Department of Electrical and Computer Engineering, The University of Texas at Austin, Austin, TX, USA  \{rjiang@utexas.edu\}} \qquad Aryan Mokhtari\thanks{Department of Electrical and Computer Engineering, The University of Texas at Austin, Austin, TX, USA  \{mokhtari@austin.utexas.edu\}}}

\date{\empty}

\maketitle

\begin{abstract}

In this paper, we explore the non-asymptotic global convergence rates of the Broyden-Fletcher-Goldfarb-Shanno (BFGS) method implemented with exact line search. Notably, due to Dixon's equivalence result, our findings are also applicable to other quasi-Newton methods in the convex Broyden class employing exact line search, such as the Davidon-Fletcher-Powell (DFP) method. Specifically, we focus on problems where the objective function is strongly convex with Lipschitz continuous gradient and Hessian. Our results hold for any initial point and any symmetric positive definite initial Hessian approximation matrix. The analysis unveils a detailed three-phase convergence process, characterized by distinct linear and superlinear rates, contingent on the iteration progress. Additionally, our theoretical findings demonstrate the trade-offs between linear and superlinear convergence rates for BFGS when we modify the initial Hessian approximation matrix, a phenomenon further corroborated by our numerical experiments.



\end{abstract}

\newpage

\section{Introduction}\label{sec:introduction}

In this paper, we consider the unconstrained minimization problem
\begin{equation}\label{objective_function}
    \min_{x \in \mathbb{R}^d} f(x),
\end{equation}
where $f: \mathbb{R}^d \to \mathbb{R}$ is strongly convex and twice continuously differentiable. We focus on the non-asymptotic global convergence properties of quasi-Newton methods for solving Problem \eqref{objective_function}. The core idea behind quasi-Newton methods is to mimic the update of Newton's method using only first-order information, i.e., the gradients of $f$. Specifically, the update rule at the $k$-th iteration is 
\begin{equation}\label{eq:quasi_newton_update}
    x_{k+1} = x_k - \eta_k B_k^{-1} \nabla f(x_k), 
\end{equation} 
where $\eta_k$ is the step size and $B_k \in \reals^{d\times d}$ is a matrix constructed from the gradients of $f$ to approximate the Hessian $\nabla^2{f(x_k)}$. Various quasi-Newton methods have been developed, each distinguished by its strategy for constructing the Hessian approximation $B_k$ and its inverse. The key methods among them are the Davidon-Fletcher-Powell (DFP) method \cite{davidon1959variable,fletcher1963rapidly}, the Broyden-Fletcher-Goldfarb-Shanno (BFGS) method \cite{broyden1970convergence,fletcher1970new,goldfarb1970family,shanno1970conditioning}, the Symmetric Rank-One (SR1) method \cite{conn1991convergence,khalfan1993theoretical}, and the Broyden method \cite{broyden1965class}. Notably, these quasi-Newton methods directly maintain and update the inverse matrix $B_k^{-1}$ using a constant number of matrix-vector multiplications. This results in a computational cost of $\mathcal{O}(d^2)$ per iteration and thus makes quasi-Newton methods more efficient than Newton's method, which involves computing the Hessian and solving a linear system that could incur a computational cost of $\mathcal{O}(d^3)$ per iteration. 

Compared to other first-order methods, such as gradient descent and accelerated gradient descent, the primary advantage of quasi-Newton methods is their ability to achieve Q-superlinear convergence, i.e.,
\begin{equation}
    \lim_{k \to \infty}\frac{f(x_{k + 1}) - f(x_*)}{f(x_k) - f(x_*)} = 0\qquad \text{or}\qquad  \lim_{k \to \infty}\frac{\|x_{k + 1} - x_*\|}{\|x_k - x_*\|} = 0,
\end{equation}
where $x_* \in \mathbb{R}^d$ denotes the optimal solution of Problem~\eqref{objective_function}. Specifically, \cite{broyden1973local} and \cite{dennis1974characterization} have established that both DFP and BFGS converge Q-superlinearly with unit step size $\eta_k = 1$, where the initial point $x_0$ is required to be within a local neighborhood of the optimal solution $x_*$. Later, it has also been extended to various settings \cite{griewank1982local,dennis1989convergence,yuan1991modified,al1998global,li1999globally,yabe2007local,mokhtari2017iqn,gao2019quasi}. However, these local convergence results are all \emph{asymptotic} and fail to provide an explicit convergence rate after a finite number of iterations.

Recently, there has been progress regarding \emph{non-asymptotic} local convergence analysis of quasi-Newton methods. The authors of \cite{rodomanov2020rates} showed that, if the initial point $x_0$ is in a local neighborhood of the optimal solution $x_*$ and the initial Hessian approximation matrix $B_0$ is initialized as $L I$, then BFGS with unit step size attains a local superlinear convergence rate of the form $(\frac{d L}{\mu k})^k$, where $d$ is the problem's dimension, $L$ is the Lipschitz parameter of the gradient, and $\mu$ is the strong convexity parameter. Later in \cite{rodomanov2020ratesnew}, the local convergence rate of BFGS was improved to $(\frac{d\log{(L/\mu)}}{k})^k$ under similar initial conditions. Similar local superlinear convergence analysis has also been established for the SR1 method \cite{ye2023towards}. In a concurrent work \cite{qiujiang2020quasinewton}, it was demonstrated that, if $x_0$ is in a local neighborhood of the optimal solution $x_*$ and $B_0$ is sufficiently close to the exact Hessian at the optimal solution (or selected as the exact Hessian at $x_0$), then BFGS with unit step size achieves a local superlinear rate of $(1/k)^{k/2}$, which is independent of the dimension $d$ and the condition number $L/\mu$. While these non-asymptotic results successfully characterize an explicit superlinear rate, they rely heavily on local analysis, requiring the initial point to be sufficiently close to the optimal solution $x_*$ and imposing conditions on the step size and the initial Hessian approximation matrix $B_0$. Consequently, these results cannot be directly extended to a global convergence guarantee. We discuss this issue in detail in Section~\ref{sec:discussions}.

To guarantee global convergence, quasi-Newton methods must be combined with line search or trust-region techniques. The first global result for quasi-Newton methods was derived by Powell in~\cite{powell1971convergence}, where it was established that DFP with exact line search converges globally and Q-superlinearly. Later, Dixon~\cite{Dixon} proved that all quasi-Newton methods from the convex Broyden's class generate the same iterates using exact line search, thus extending Powell's result to the convex Broyden's class including BFGS. In order to relax the exact line search condition, the work in~\cite{Powell} considered BFGS using inexact line search based on Wolfe conditions and showed that it retains global superlinear convergence. This result was later extended in~\cite{byrd1987global} to the convex Broyden class except for DFP. Moreover,~\cite{conn1991convergence,khalfan1993theoretical,byrd1996analysis} showed that the SR1 method with trust-region techniques achieves global and superlinear convergence.

However, all these results lack an explicit global convergence rate; they only provide asymptotic convergence guarantees and fail to characterize the explicit global convergence rate of classical quasi-Newton methods. The only exception is a recent work in \cite{krutikov2023convergence}, where the authors also studied the global convergence rate of BFGS with exact line search. Specifically, it was shown that BFGS attains a global linear rate of $(1 - \frac{2\mu^3}{L^3} (1 + \frac{\mu \Tr(B_0^{-1})}{k})^{-1}(1 + \frac{\Tr(B_0)}{Lk})^{-1})^k$, where $\Tr(\cdot)$ denotes the trace of a matrix. We note that after $k = O(d)$ iterations, their linear rate approaches the rate of $(1-\frac{2\mu^3}{L^3})^k$, which is substantially slower than gradient descent-type methods. More importantly, their study does not extend to demonstrating any superlinear convergence rate and fails to fully characterize the behavior of BFGS.

The discussions above reveal a major gap in classical quasi-Newton methods: the lack of an explicit global convergence rate characterization.

\vspace{2mm}

\noindent\textbf{Contributions.} In this paper, we present the first results that contain explicit non-asymptotic global linear and superlinear convergence rates for the BFGS method with exact line search. Note that due to the equivalence result by Dixon \cite{Dixon}, our results also hold for other quasi-Newton methods in the convex Broyden class with exact line search. At a high level, our convergence analysis sharpens the potential function-based framework first introduced in~\cite{QN_tool}, leading to a unifying framework for proving both the global linear convergence rates and the superlinear convergence rates. Our convergence results are global as they hold for any initial point $x_0 \in \mathbb{R}^d$ and any initial Hessian approximation matrix $B_0$ that is symmetric positive definite. Specifically, our analysis divides the convergence process into three phases, characterized by different convergence rates:
\begin{enumerate}[(i)]
    \item First linear phase: We show that  
    \begin{equation*}
        \frac{f(x_k) - f(x_*)}{f(x_0) - f(x_*)} \leq \left(1 - e^{-\frac{\Psi(\bar{B}_0)}{k}}\frac{1}{\kappa}\max\left\{\frac{2}{1 + \sqrt{\kappa}}, \frac{1}{1 + C_0}\right\}\right)^{k}.
    \end{equation*}
    {Here, $\bar{B}_0 = \frac{1}{L} B_0$ is the scaled initial Hessian approximation matrix, $\Psi(\cdot)$ is a potential function defined later in \eqref{potential_function}, $\kappa = \frac{L}{\mu}$ denotes the condition number, and $C_0 = \frac{2M\sqrt{2(f(x_0) - f(x_*))}}{\mu^{{3}/{2}}}$ is based on the initial optimality gap with $M$ as the Hessian's Lipschitz parameter. In particular, when $k \geq \Psi(\bar{B}_0)$, this leads to a linear rate of }
    \begin{equation*}
        \frac{f(x_k) - f(x_*)}{f(x_0) - f(x_*)} \leq \left(1 - \frac{1}{3\kappa}\max\left\{\frac{2}{1 + \sqrt{\kappa}}, \frac{1}{1 + C_0}\right\}\right)^{k}. 
    \end{equation*}
    \item Second linear phase: Upon reaching $k \geq (1+C_0) \Psi(\bar{B}_0) + 3C_0 \kappa \min\{2(1 + C_0), 1+\sqrt{\kappa}\}$, the algorithm attains an improved linear rate matching that of standard gradient descent:  
    $$\frac{f(x_k) - f(x_*)}{f(x_0) - f(x_*)} \leq \left(1-\frac{1}{3\kappa}\right)^k.$$ 
    \item Superlinear phase: when $ k \geq \Psi(\Tilde{B}_{0}) + 4C_0\Psi(\bar{B}_0) + 12C_0 \kappa\min\{2(1 + C_0), 1 +\sqrt{\kappa}\}$, BFGS achieves a superlinear convergence rate of 
    \begin{equation*}
        \frac{f(x_k) - f(x_*)}{f(x_0) - f(x_*)} \leq \left(\frac{\Psi(\Tilde{B}_{0}) + 4C_0\Psi(\bar{B}_0) + 12C_0 \kappa\min\{2(1 + C_0), 1 +\sqrt{\kappa}\}}{k}\right)^{k},
    \end{equation*}
    where $\tilde{B}_0 = \nabla^2f(x_*)^{-\frac{1}{2}}B_0\nabla^2f(x_*)^{-\frac{1}{2}}$ is the normalized initial Hessian approximation matrix. 
\end{enumerate}

\begin{table}[t]
  \centering
  \caption{Summary of our convergence results. The last column presents the number of iterations required to achieve the corresponding linear or superlinear convergence phase. For brevity, we drop absolute constants in our results.}
  \vspace{2mm}
  \label{tab:comparison}
  \begin{tabular}{ |c|c|c|c| }
    \hline
    $B_0$ & Convergence Phase & Convergence Rate & Starting moment \\
    \hline
    \hline
    $\alpha I$ & Linear phase I & $ \left(1-\frac{1}{\kappa \min\{1 + C_0, \sqrt{\kappa}\}}\right)^{k}$ & $d(\frac{\alpha}{L} - 1 + \log{\frac{L}{\alpha}})$ \\
    \hline
    $\alpha I$ & Linear phase II & $\left(1-\frac{1}{\kappa}\right)^{k}$  & \begin{tabular}{@{}c@{}} $ (1 + C_0)d(\frac{\alpha}{L} - 1 + \log{\frac{L}{\alpha}})$ \\ $ + C_0\kappa\min\{1 + C_0, \sqrt{\kappa}\}$ \end{tabular} \\
    \hline
    $\alpha I$ & Superlinear phase & \begin{tabular}{@{}c@{}}$  (d(\frac{\alpha}{\mu} - 1 + \log{\frac{L}{\alpha}})/k $ \\ $+ C_0d(\frac{\alpha}{L} - 1 + \log{\frac{L}{\alpha}})/k $ \\ $ + C_0 \kappa \min\{1 + C_0, \sqrt{\kappa}\}/k)^{k} $\end{tabular} & \begin{tabular}{@{}c@{}}$  d(\frac{\alpha}{\mu} - 1 + \log{\frac{L}{\alpha}}) $ \\ $+ C_0d(\frac{\alpha}{L} - 1 + \log{\frac{L}{\alpha}})  $ \\ $ + C_0 \kappa \min\{1 + C_0, \sqrt{\kappa}\}$\end{tabular} \\
    \hline
    $L I$ & Linear phase I & $ \left(1-\frac{1}{\kappa \min\{1 + C_0, \sqrt{\kappa}\}}\right)^{k}$ & $1$ \\
    \hline
    $L I$ & Linear phase II & $\left(1-\frac{1}{\kappa}\right)^{k}$  & $ C_0 \kappa \min\{1 + C_0, \sqrt{\kappa}\}$ \\
    \hline
    $L I$ & Superlinear phase & $\left(\frac{d\kappa + C_0\kappa\min\{1 + C_0, \sqrt{\kappa}\}}{k}\right)^{k}$ & \begin{tabular}{@{}c@{}}$  d\kappa + $ \\ $ C_0 \kappa \min\{1 + C_0, \sqrt{\kappa}\}$\end{tabular} \\
    \hline
    $\mu I$ & Linear phase I & $\left(1-\frac{1}{\kappa \min\{1 + C_0, \sqrt{\kappa}\}}\right)^{k}$ & $ d\log\kappa$ \\
    \hline
    $\mu I$ & Linear phase II & $\left(1-\frac{1}{\kappa}\right)^{k}$  & \begin{tabular}{@{}c@{}}$ (1 + C_0)d\log\kappa + $ \\ $C_0 \kappa \min\{1 + C_0, \sqrt{\kappa}\}$\end{tabular} \\
    \hline
    $\mu I$ & Superlinear phase &\!\! $\!\!\left(\frac{(1 + C_0)d\log\kappa + C_0\kappa\min\{1 + C_0, \sqrt{\kappa}\}}{k}\right)^{k}\!\!$ & \begin{tabular}{@{}c@{}}$  (1 + C_0)d\log\kappa + $ \\ $C_0 \kappa \min\{1 + C_0, \sqrt{\kappa}\}$\end{tabular} \\
    \hline
    $c I$ & Linear phase I & $ \left(1-\frac{1}{\kappa \min\{1 + C_0, \sqrt{\kappa}\}}\right)^{k}$ & $d\log\kappa$ \\
    \hline
    $c I$ & Linear phase II & $\left(1-\frac{1}{\kappa}\right)^{k}$  & \begin{tabular}{@{}c@{}}$  (1 + C_0)d\log\kappa + $ \\ $C_0 \kappa \min\{1 + C_0, \sqrt{\kappa}\}$\end{tabular} \\
    \hline
    $c I$ & Superlinear phase & $\left(\frac{d\kappa + C_0 d\log\kappa + C_0\kappa\min\{1 + C_0, \sqrt{\kappa}\}}{k}\right)^{k}$ & \begin{tabular}{@{}c@{}}$  d\kappa + C_0 d \log\kappa $ \\ $ C_0 \kappa \min\{1 + C_0, \sqrt{\kappa}\}$\end{tabular} \\
    \hline
  \end{tabular}
\end{table}

To make our convergence rates easily interpretable, we focus on the global linear and superlinear convergence rates of the special case where $B_0 = \alpha I$ for a given scalar $\alpha > 0$. We also study the practical initialization, $B_0 = cI$,  where $c = \frac{s^\top y}{\|s\|^2}$ with \( s = x_2 - x_1 \), \( y = \nabla f(x_2) - \nabla f(x_1) \), and \( x_1 \), \( x_2 \) being two randomly chosen vectors. We further consider $B_0 = LI$ and $B_0 = \mu I$ as two specific cases. The global convergence results with these initializations are summarized in Table~\ref{tab:comparison}. Our analysis reveals a trade-off between the linear and the superlinear rates, depending on the choice of the initial matrix $B_0$. Specifically, while both initializations lead to the same linear convergence rates, initiating with $B_0 = L I$ allows the algorithm to reach this rate $d \log \kappa$ iterations earlier than with $B_0 = \mu I$. On the other hand, for the superlinear convergence phase, the difference between $B_0 = L I$ and $B_0 = \mu I$ essentially boils down to comparing $d\kappa$ against $(1+4C_0)d\log \kappa$.  Thus, when $C_0  \ll \kappa$, initializing with $B_0 = \mu I$ enables an earlier transition to the superlinear convergence compared to $B_0 = L I$, as well as a faster superlinear convergence rate. As we shall see in Section~\ref{sec:experiments}, our experiments also demonstrate this trade-off.
\vspace{2mm}

\noindent\textbf{Additional related work.} In addition to the standard quasi-Newton methods such as BFGS, the superlinear convergence of other variants of quasi-Newton methods has also been studied in the literature. The greedy variants of quasi-Newton methods were first introduced in \cite{rodomanov2020greedy} and developed in subsequent works \cite{lin2021greedy,lin2022explicit,ji2023greedy}. 
Instead of using the difference of successive iterates to update the Hessian approximation matrix, the key idea is to greedily select basis vectors to maximize a certain measure of progress. In \cite{rodomanov2020greedy}, greedy BFGS is shown to achieve a local superlinear convergence rate of $(d\kappa(1 - \frac{1}{d\kappa})^\frac{k}{2})^{k}$ and the superlinear convergence phase begins after $d\kappa\ln{(d\kappa)}$ iterations. Similar superlinear convergence rates are extended to other greedy quasi-Newton updates in \cite{lin2021greedy,lin2022explicit,ji2023greedy}.  However, we note that their results are all local and require the initial point to be sufficiently close to the optimal solution $x_*$. Recently, along a different line of work, the authors in \cite{jiang2023online,jiang2023accelerated} proposed quasi-Newton-type methods based on the hybrid proximal extragradient framework~\cite{solodov1999hybrid,monteiro2010complexity} and studied their global convergence rates. Specifically, it was shown that the quasi-Newton proximal extragradient method in \cite{jiang2023online} achieves a global linear convergence rate of $(1-{1}/{\kappa})^k$ and a global superlinear rate of the form $(1+\sqrt{k/\bigO(\kappa^2 d)})^{-k}$. However, these methods are distinct from the classical quasi-Newton methods such as BFGS analyzed in this paper, since they formulate the update of the Hessian approximation matrices $B_k$ as an online convex optimization problem and follow an online learning algorithm to update $B_k$.

\vspace{2mm}

\noindent\textbf{Outline.} In Section~\ref{sec:preliminaries}, we provide an overview of the BFGS method with exact line search, outline our assumptions, and introduce some preliminary lemmas for the exact line search scheme. Section~\ref{sec:basic} presents our general analytical framework, which is employed to establish global linear and superlinear convergence results for the BFGS method, along with the intermediate results for the update of quasi-Newton methods. {In Section~\ref{sec:linear}, we establish the global linear convergence rate of BFGS using exact line search that applies to any choices of $B_0$ and $x_0$ and we consider specific initializations with $B_0 = \alpha I$, $B_0 = cI$, $B_0=LI$ and $B_0=\mu I$.} Building on the linear convergence rates, Section~\ref{sec:superlinear} details our global superlinear convergence results. In Section~\ref{sec:discussions}, we compare our analytical framework to both classical asymptotic analysis and recent local non-asymptotic analysis of BFGS. Section~\ref{sec:experiments} displays our numerical experiments that corroborate our theoretical findings. Finally, we finish the paper by presenting some concluding remarks in Section~\ref{sec:conclusion}.

\vspace{2mm}

\noindent\textbf{Notation.} We use $\|\cdot\|$ to denote the $\ell_2$-norm of a vector or the spectral norm of a matrix. We denote $\mathbb{S}^d_{+}$ and $\mathbb{S}^d_{++}$ as the set of symmetric positive semidefinite and symmetric positive definite matrices with dimension $d \times d$, respectively. Given two symmetric matrices $A$ and $B$, we denote $A \preceq B$ if and only if $B - A$ is positive semidefinite. Given a matrix $A$, we use $\Tr(A)$ and $\Det(A)$ to denote its trace and determinant, respectively. 

\section{Preliminaries}\label{sec:preliminaries}

In this section, we first outline the assumptions, notations, and lemmas essential for our convergence proof. Following this, we explore the general framework of quasi-Newton methods incorporating exact line search and provide an overview of the principal concepts underpinning the update mechanism in the convex Broyden’s class of quasi-Newton methods, which encompasses both the BFGS and DFP algorithms.

\subsection{Assumptions}

To begin with, we state our assumptions on the objective functions $f$. 
\begin{assumption}\label{ass_str_cvx}
The objective function $f$ is strongly convex with parameter $\mu > 0$, i.e.,  $\|\nabla{f(x)} - \nabla{f(y)}\| \geq \mu\| x - y\|$, for any $x, y \in \mathbb{R}^{d}$.
\end{assumption}
\begin{assumption}\label{ass_smooth}
The objective function gradient $\nabla f$ is Lipschitz continuous with parameter $L > 0$, i.e., $\|\nabla{f(x)} - \nabla{f(y)}\| \leq L\| x - y\|$ for any $x, y \in \mathbb{R}^{d}$.
\end{assumption}
Both Assumptions~\ref{ass_str_cvx} and~\ref{ass_smooth} are standard in the convergence analysis of first-order methods. Moreover, since $f$ is twice differentiable, they imply that $\mu I \preceq \nabla^2{f(x)} \preceq LI$ for any $x \in \mathbb{R}^{d}$.  Additionally, the condition number of $f$ is defined as $\kappa := \frac{L}{\mu}$. We also remark that Assumptions~\ref{ass_str_cvx} and~\ref{ass_smooth} are sufficient in our analysis to prove a global linear convergence rate of BFGS with exact line search. In order to achieve a superlinear convergence rate, we need to impose an additional assumption on the Hessian of the function $f$, stated below.  

\begin{assumption}\label{ass_Hess_lip}
The objective function Hessian $\nabla^2 f$ is Lipschitz continuous along the direction of optimal solution $x_*$ with parameter $M > 0$, i.e., $\|\nabla^{2}{f(x)} - \nabla^{2}{f(x_*)}\| \leq M\|x - x_*\|$ for any $x \in \mathbb{R}^{d}$.
\end{assumption}
Assumption~\ref{ass_Hess_lip} is commonly used in the analysis of quasi-Newton methods, such as in \cite{QN_tool}, as it provides a necessary smoothness condition for the Hessian of the objective function. Importantly, we do not require Hessian smoothness for arbitrary points \( x, y \in \mathbb{R}^d \); rather, we impose a weaker condition, assuming that the Hessian is Lipschitz continuous only along the direction of the optimal solution \( x_* \).

\subsection{Quasi-Newton methods with exact line search}

{Next, we briefly review the template for updating quasi-Newton matrices, focusing specifically on the DFP and BFGS algorithms. }Specifically, at the $k$-th iteration, the update in \eqref{eq:quasi_newton_update} can be equivalently written as 
\begin{equation}\label{QN_update}
    x_{k+1} = x_k + \eta_k d_k, \qquad  \text{where }\; d_k = - B_k^{-1} g_k \quad \text{and} \quad g_k = \nabla{f(x_k)}.
\end{equation}
Here, $\eta_k \geq 0$ represents the step size, and $B_k \in \mathbb{R}^{d\times d}$ is the Hessian approximation matrix. Replacing $B_k$ with the exact Hessian $\nabla^2f(x_k)$ turns the update into the classical Newton's method. Quasi-Newton methods aim to approximate the Hessian with first-order information, typically adhering to a \emph{secant condition} and a \emph{least-change property}. To elaborate, we define the variable difference $s_k$ and gradient difference $y_k$ as  
\begin{equation}\label{difference}
      s_k := x_{k+1} - x_k, \qquad y_k := \nabla f(x_{k+1}) - \nabla f(x_k).
\end{equation}
The secant condition requires that $B_{k+1}$ satisfies $y_{k} = B_{k+1} s_k$, ensuring the gradient consistency between the quadratic model $h_{k+1}(x) = f(x_{k+1}) + g_{k+1}^\top (x-x_{k+1}) + \frac{1}{2}(x-x_{k+1})^\top B_{k + 1} (x-x_{k+1})$ and $f$ at $x_{k}$ and $x_{k+1}$; that is, $\nabla h_{k+1}(x_k) = \nabla f(x_k)$ and $\nabla h_{k+1}(x_{k+1}) = \nabla f(x_{k+1})$ (see \cite[Chapter 6]{nocedal2006numerical}). That said, the secant condition does not uniquely define $B_{k+1}$. Thus, we impose a least-change property to ensure that $B_{k+1}$, satisfying the secant condition, is closest to $B_k$ in a specific proximity measure. Various proximity measures have been proposed in the literature \cite{goldfarb1970family,greenstadt1970variations,fletcher1991new} and here we follow the variational characterization in~\cite{fletcher1991new}. Specifically, for any symmetric positive definite matrix $A \in \mathbb{S}^d_{++}$, define the negative log-determinant function $\Phi(A) = -\log \Det (A)$ and define the Bregman divergence generated by $\Phi$ by
\begin{equation}\label{eq:Bregman}
\begin{split}
    D_{\Phi}(A,B) & := \Phi(A) - \Phi(B) - \langle \nabla \Phi(B), A-B \rangle \\
    & \phantom{:}= \Tr(B^{-1}A) - \log\Det(B^{-1}A) - d.
\end{split}
\end{equation}
Note that the Bregman divergence can be regarded as a measure of proximity between two positive definite matrices, and  $D_{\Phi}(A,B) = 0$ if and only if $A = B$. For the BFGS update, it was shown in~\cite{fletcher1991new} that $B_{k+1}$ is given as the unique solution of the minimization problem: 
\begin{equation*}
    \min_{B \in \mathbb{S}^d_{++}}\; D_{\Phi}(B; B_k ) \quad \text{s.t.}\quad y_k = B s_k, 
\end{equation*}
which admits the following explicit update rule:
\begin{equation}\label{BFGS_Hessian_update}
    B^{\text{BFGS}}_{k+1} := B_k - \frac{B_k s_k s_k^\top B_k}{s_k^\top B_k s_k} + \frac{y_k y_k^\top}{s_k^\top y_k}.
\end{equation} 
{Moreover, if we define $H_k := B_k^{-1}$ as the inverse of the Hessian approximation matrix, it follows from the Sherman-Morrison-Woodbury formula that }
\begin{equation}\label{BFGS_Hessian_inverse_update}
    H^{\text{BFGS}}_{k+1} := \left(I-\frac{s_k y_k^\top}{y_k^\top s_k}\right) H_k \left(I-\frac{ y_k s_k^\top}{s_k^\top y_k}\right) +\frac{s_k s_k^\top}{y_k^\top s_k}.
\end{equation}
The DFP update rule can be regarded as the dual of BFGS, where the roles of the Hessian approximation matrix $B_{k+1}$ and its inverse $H_{k+1}$ are exchanged. Specifically, the DFP update rules are given by 
\begin{align*}
    & B^{\text{DFP}}_{k+1} := \left(I-\frac{y_k s_k^\top}{y_k^\top s_k}\right) B_k \left(I-\frac{ s_ky_k^\top}{s_k^\top y_k}\right) +\frac{y_k y_k^\top}{y_k^\top s_k}, \\
    & H^{\text{DFP}}_{k+1} := H_k - \frac{H_k y_k y_k^\top H_k}{y_k^\top H_k y_k} + \frac{s_k s_k^\top}{s_k^\top y_k}.
\end{align*}
Both BFGS and DFP belong to a more general class of quasi-Newton methods, known as the convex Broyden's class~\cite{broyden1967quasi}. In this class, the Hessian approximation matrix $B_{k + 1}$ is defined as 
\begin{equation*}
    B_{k+1} := \phi_k B^{\text{DFP}}_{k+1} + (1 - \phi_k) B^{\text{BFGS}}_{k+1},
\end{equation*}
where $\phi_k \in [0,1]$ for any $k \geq 0$. Accordingly, there exists $\psi_k \in [0,1]$ such that the Hessian inverse approximation matrix $H_{k + 1}$ is given by
\begin{equation*}
    H_{k+1} := (1 - \psi_k) H^{\text{DFP}}_{k+1} + \psi_k H^{\text{BFGS}}_{k+1}.
\end{equation*}
The convex Broyden's class exhibits a crucial property: if the initial Hessian approximation matrix $B_0$ is symmetric positive definite and the objective function $f$ is strictly convex, then all subsequent $B_k$ matrices produced by this class maintain symmetric positive definiteness (see \cite{nocedal2006numerical}).

To guarantee the global convergence of quasi-Newton methods in \eqref{QN_update}, it is necessary to employ a line search scheme to select the step size $\eta_k$. In this paper, our primary focus is on the exact line search step size, where we aim to minimize the objective function along the search direction $d_k$. Specifically,
\begin{equation}\label{exact_line_search}
    \eta_k := \argmin_{\eta \geq 0} f(x_k + \eta d_k).
\end{equation}
Remarkably, it was shown in \cite{Dixon} that, when employing the exact line search scheme, the convex Broyden's class of quasi-Newton methods produce identical iterates given that the initial point $x_0$ and the initial matrix $B_0$ are the same. Thus, in the remainder of the paper, we focus on the BFGS update in \eqref{BFGS_Hessian_update} as all results hold for other algorithms in the convex Broyden family.

Finally, we introduce some intermediate results related to the exact line search step size, as defined in \eqref{exact_line_search}. These standard results (see, e.g., \cite{powell1971convergence}) are essential for the forthcoming demonstration of the convergence rate of the quasi-Newton method.

\begin{lemma}\label{lemma_line_search}
    Consider the standard quasi-Newton method in \eqref{QN_update} with the exact line search specified in \eqref{exact_line_search}. The following results hold for any $k \geq 0$: 
    \begin{enumerate}[(a)] 
        \item $f(x_{k+1}) \leq f(x_k)$. 
        \item \label{item:gs=0} $g_{k + 1}^\top s_k = 0$ and $y_k^{\top} s_k = - g_k^{\top} s_k$. 
    \end{enumerate}
\end{lemma}
    

\section{Convergence analysis framework}\label{sec:basic}

In this section, we introduce our theoretical framework for establishing the global convergence rates of the BFGS algorithm with exact line search. Our framework builds on two key propositions. In Proposition~\ref{lemma_bound}, we characterize the amount of function value decrease in one iteration in terms of the angle $\theta_k$ between the steepest descent direction $-g_k$ and the search direction $d_k$ given in \eqref{QN_update}. Subsequently, Proposition~\ref{lemma_BFGS} presents a potential function for the BFGS update,  which leads to a lower bound on $\cos(\theta_k)$.

To formally begin the analysis, we first present a weighted version of key vectors and matrices as introduced from the previous work \cite{dennis1974characterization}. Specifically, given a weight matrix $P \in \mathbb{S}_{++}^d$ (also referred to as a transformation matrix), we define the weighted gradient $\hat{g}_k$, the weighted gradient difference $\hat{y}_k$, and the weighted iterate difference $\hat{s}_k$ as
\begin{equation}\label{weighted_vector}
    \hat{g}_k = P^{-\frac{1}{2}}g_k, \qquad \hat{y}_k = P^{-\frac{1}{2}} y_k, \qquad \hat{s}_k = P^{\frac{1}{2}}s_k.
\end{equation} 
Similarly, we define the weighted Hessian approximation matrix $\hat{B}_k$ as 
\begin{equation}\label{weighted_matrix_1}
    \hat{B}_k = P^{-\frac{1}{2}} {B}_k  P^{-\frac{1}{2}}.
\end{equation}
Note that the weight matrix $P$ can be chosen as any positive definite matrix, and its choice will be evident from the context.  In particular, as we shall see later, we use $P = L I$ in Section~\ref{sec:linear} to prove the global linear convergence rate, and use $P = \nabla^2 f(x_*)$ in Section~\ref{sec:superlinear} to prove the global superlinear convergence rate. Moreover, since the above weighting procedure amounts to a change of the coordinate system, the weighted versions of the vectors and matrices defined in \eqref{weighted_vector} and \eqref{weighted_matrix_1} retain the same algebraic relations as their original forms. In particular, the weighted Hessian approximation matrices generated by the BFGS algorithm follow the subsequent update rule:
\begin{equation}\label{BFGS_weighted}
    \hat{B}_{k+ 1} = \hat{B}_k - \frac{\hat{B}_k \hat{s}_k \hat{s}_k^\top \hat{B}_k}{\hat{s}_k^\top \hat{B}_k \hat{s}_k} + \frac{\hat{y}_k \hat{y}_k^\top}{\hat{s}_k^\top \hat{y}_k}. 
\end{equation}
Before introducing our first key proposition, we define a quantity $\hat{\theta}_k$ by  
\begin{equation}\label{eq:theta}
    \cos(\hat{\theta}_k) = \frac{-\hat{g}_k^\top \hat{s}_k}{\|\hat{g}_k\|\|\hat{s}_k\|},
\end{equation}
which is the angle between the weighted steepest descent direction $-\hat{g}_k$ and the weighted iterate difference $\hat{s}_k$. It is well-known that the convergence of quasi-Newton methods can be established by monitoring the behavior of $\cos(\hat{\theta}_k)$. We next quantify the link between functional value decrease and $\cos(\hat{\theta}_k)$.

\begin{proposition}\label{lemma_bound}
    Let $\{x_k\}_{k\geq 0}$ be the iterates generated by the BFGS method with exact line search. Given a weight matrix $P\in \mathbb{S}^d_{++}$, recall the weighted vectors and matrices defined in \eqref{weighted_vector} and \eqref{weighted_matrix_1}.  For any $k \geq 0$, we have
    \begin{equation}\label{eq:one_step_bound}
        f(x_{k+1}) -f(x_*) = \left(1-\frac{\hat{\alpha}_{k} \hat{q}_{k}}{\hat{m}_{k}} \cos^2(\hat{\theta}_{k}) \right)(f(x_{k})-f(x_*)),  
    \end{equation}
    where we define 
    \begin{equation}\label{eq:def_alpha_q_m}
        \hat{\alpha}_k := \frac{f(x_k)-f(x_{k+1})}{-\hat{g}_k^\top \hat{s}_k}, \qquad \hat{q}_k := \frac{\|\hat{g}_k\|^2}{f(x_k)-f(x_*)}, \qquad \hat{m}_k := \frac{\hat{y}_k^\top \hat{s}_k}{\|\hat{s}_k\|^2}. 
    \end{equation}
    As a corollary, we have that for any $k \geq 1$,
    \begin{equation}\label{eq:product}
        \frac{f(x_{k}) - f(x_*)}{f(x_0) - f(x_*)} \leq \left[1 - \left(\prod_{i = 0}^{k-1} \frac{\hat{\alpha}_i \hat{q}_i}{\hat{m}_i} \cos^2(\hat{\theta}_i)\right)^{\frac{1}{k}}\right]^{k}.
    \end{equation}
\end{proposition}

\begin{proof}
    First, we use the definition of $ \hat{\alpha}_k$ in \eqref{eq:def_alpha_q_m} to write
    \begin{equation}\label{eq:rewrite_alpha}
        f(x_k) - f(x_{k+1})  = -\hat{\alpha}_k \hat{g}_k^\top  \hat{s}_k = -\hat{\alpha}_k \frac{\hat{g}_k^\top  \hat{s}_k}{\|\hat{g}_k\|^2} \|\hat{g}_k\|^2.
    \end{equation}
    Moreover, note that we have $-\hat{g}_k^\top  \hat{s}_k = \hat{y}_k^\top \hat{s}_k$ by Lemma~\ref{lemma_line_search}(b). Hence, using the definition of $\hat{\theta}_k$ in \eqref{eq:theta} and the definition of $\hat{m}_k$ in \eqref{eq:def_alpha_q_m}, it follows that 
    \begin{equation*}
        \frac{-\hat{g}_k^\top  \hat{s}_k}{\|\hat{g}_k\|^2} = \frac{(\hat{g}_k^\top  \hat{s}_k)^2}{\|\hat{g}_k\|^2\|\hat{s}_k\|^2} \frac{\|\hat{s}_k\|^2}{-\hat{g}_k^\top  \hat{s}_k} = \frac{(\hat{g}_k^\top  \hat{s}_k)^2}{\|\hat{g}_k\|^2\|\hat{s}_k\|^2} \frac{\|\hat{s}_k\|^2}{\hat{y}_k^\top  \hat{s}_k} = \frac{\cos^2(\hat{\theta}_k)}{\hat{m}_k}. 
    \end{equation*}
    Furthermore, we have $\|\hat{g}_k\|^2 = \hat{q}_k (f(x_k)-f(x_*))$ from the definition of $\hat{q}_k$ in \eqref{eq:def_alpha_q_m}. Thus, the equality in \eqref{eq:rewrite_alpha} can be rewritten as
    \begin{align*}
        f(x_k) - f(x_{k+1})  = \frac{\hat{\alpha}_k \hat{q}_k}{\hat{m}_k} \cos^2(\hat{\theta}_k)(f(x_k)-f(x_*)).
    \end{align*}
    By rearranging the term in the above equality, we obtain \eqref{eq:one_step_bound}. To prove the inequality in \eqref{eq:product}, note that for any $k \geq 1$, we have 
    \begin{equation*}
        \frac{f(x_{k}) - f(x_*)}{f(x_0) - f(x_*)} = \prod_{i = 0}^{k - 1}\frac{f(x_{i + 1}) - f(x_*)}{f(x_i) - f(x_*)} = \prod_{i = 0}^{k - 1}\left(1 - \frac{\hat{\alpha}_i \hat{q}_i}{\hat{m}_i} \cos^2(\hat{\theta}_i)\right), 
    \end{equation*}
    where the last equality is due to \eqref{eq:one_step_bound}. Notice that the term $1 - \frac{\hat{\alpha}_i \hat{q}_i}{\hat{m}_i} \cos^2(\hat{\theta}_i)$ are non-negative for any $i \geq 0$. Thus, by applying the inequality of arithmetic and geometric means twice, we obtain that  
    \begin{equation*}
    \begin{aligned}
        & \prod_{i = 0}^{k - 1}\left(1 - \frac{\hat{\alpha}_i \hat{q}_i}{\hat{m}_i} \cos^2(\hat{\theta}_i)\right) \leq \left[\frac{1}{k}\sum_{i = 0}^{k - 1}\left(1 - \frac{\hat{\alpha}_i \hat{q}_i}{\hat{m}_i} \cos^2(\hat{\theta}_i)\right)\right]^{k} \\
        & = \left[1 - \frac{1}{k}\sum_{i = 0}^{k - 1}\frac{\hat{\alpha}_i \hat{q}_i}{\hat{m}_i} \cos^2(\hat{\theta}_i)\right]^{k} \leq \left[1 - \left(\prod_{i = 0}^{k - 1} \frac{\hat{\alpha}_i \hat{q}_i}{\hat{m}_i} \cos^2(\hat{\theta}_i)\right)^{\frac{1}{k}}\right]^{k}.
    \end{aligned}
    \end{equation*}
    This completes the proof. 
\end{proof}

\begin{remark}\label{remark_1}
    We note that similar results relating $f(x_k)-f(x_{k+1})$ to $ \cos^2(\hat{\theta}_k)$ have appeared in prior work such as \cite[Lemma 4.2]{byrd1987global} and \cite{QN_tool}, though they are used in the analysis of quasi-Newton methods with inexact line search. Compared with these prior results, Proposition~\ref{lemma_bound} is more general in the sense that we consider the weighted iterates using a general weight matrix $P$. This flexibility enables us to obtain tighter bounds and, more importantly, to obtain a global superlinear convergence rate under the same framework (see Section~\ref{sec:superlinear}). Another subtle yet important difference is that previous works typically upper bound the term $\hat{m}_k$ by $L$ prematurely, leading to the worst dependence on the condition number $\kappa$. Instead, we keep $\hat{m}_k$ in \eqref{eq:one_step_bound} as is  and lower bound the term $\cos^2(\hat{\theta}_k)/\hat{m}_k$ together,  as later shown in Proposition~\ref{lemma_BFGS}. 
\end{remark}

\begin{remark}
    Without loss of generality, we assume that $x_k \neq x_*$ for any $k \geq 0$ throughout the paper, meaning the exact optimal solution $x_*$ is never reached.
    Consequently, this ensures that $\hat{g}_k \neq 0$, $\eta_k >0$, and $\hat{s}_k \neq 0$ for any $k \geq 0$. As a result, we have $\hat{g}_k^\top \hat{s}_k = \eta_k \hat{g}_k^\top \hat{B}_k^{-1}\hat{g}_k \neq 0$, since $\hat{B}_k$ is symmetric positive definite. Moreover, under this assumption, the definitions in equation (15) are valid since all the denominators---$-\hat{g}_k^\top \hat{s}_k$, $f(x_k) - f(x_*)$, and $\|\hat{s}_k\|^2$---are nonzero. Similarly, $\hat{m}_k = \frac{\hat{y}_k^\top \hat{s}_k}{\|\hat{s}_k\|^2}$ is well-defined and nonzero, as $\hat{y}_k^\top \hat{s}_k = - \hat{g}_k^\top \hat{s}_k \neq 0$. 
\end{remark}

Proposition~\ref{lemma_bound} shows that BFGS's convergence rate hinges on four quantities: $\hat{\alpha}_k$, $\hat{q}_k$, $\hat{m}_k$, and $\cos(\hat{\theta}_k)$. Note that $\hat{\alpha}_k$ and $\hat{q}_k$ can be bounded using Assumptions~\ref{ass_str_cvx}, ~\ref{ass_smooth} and \ref{ass_Hess_lip}, independent of the quasi-Newton update, with details deferred to Section~\ref{subsec:intermediate}. The focus here is to establish a lower bound for $\cos^2(\hat{\theta}_k)/\hat{m}_k$. This involves analyzing the dynamics of the Hessian approximation matrices $\{B_k\}_{k\geq 0}$ through their trace and determinant, leveraging the following potential function from \cite{QN_tool} that integrates both:
\begin{equation}\label{potential_function}
    \Psi(A) := \mathbf{Tr}(A) - \log{\mathbf{Det}(A)} - d.
\end{equation}
Given \eqref{eq:Bregman}, $\Psi(A)$ can be regarded as the Bregman divergence generated by  $\Phi(A) = -\log \det (A) $ between the matrix $A$ and the identity matrix $ I$. In particular, $\Psi(A) \geq 0$ and also we have $\Psi(A) = 0$ if and only if $A = I$. Now we are ready to state Proposition~\ref{lemma_BFGS}, which is a classical result in the quasi-Newton literature (e.g., see \cite[Section 6.4]{nocedal2006numerical}). For completeness, we provide its proof in Appendix~\ref{appen:lemma_BFGS}.   

\begin{proposition}\label{lemma_BFGS}
    Given a weight matrix $P\in \mathbb{S}^d_{++}$, recall the weighted vectors and matrices defined in \eqref{weighted_vector} and \eqref{weighted_matrix_1}.  Let $\{\hat{B}_k\}_{k\geq 0}$ be the weighted Hessian approximation matrices generated by the BFGS update in \eqref{BFGS_weighted}. Then we have
    \begin{equation}\label{eq:potential_decrease}
        \Psi(\hat{B}_{k+1}) \leq \Psi(\hat{B}_k) + \frac{\|\hat{y}_k\|^2}{\hat{s}_k^\top \hat{y}_k} -1 + \log \frac{\cos^2\hat{\theta}_k}{\hat{m}_k}, \qquad \forall k \geq 0,
    \end{equation}
    where $\hat{m}_k$ and $\hat{\theta}_k$ are defined in \eqref{eq:def_alpha_q_m}. As a corollary, we have for any $k \geq 1$,
    \begin{equation}\label{eq:sum_of_logs}
        \sum_{i = 0}^{k - 1} \log{\frac{\cos^2(\hat{\theta}_i)}{\hat{m}_i}} \geq  - \Psi(\hat{B}_{0}) + \sum_{i = 0}^{k - 1}\left(1-\frac{\|\hat{y}_i\|^2}{\hat{s}_i^\top \hat{y}_i} \right).
    \end{equation}
\end{proposition}

Taking exponentiation of both sides in \eqref{eq:sum_of_logs}, Proposition~\ref{lemma_BFGS} provides a lower bound for the product $\prod_{i=0}^{k - 1} \frac{\cos^2(\hat{\theta}_i)}{\hat{m}_i}$ in relation to the sum $\sum_{i=0}^{k - 1} \frac{\|\hat{y}_i\|^2}{\hat{s}_i^\top \hat{y}_i}$ and $\Psi(\hat{B}_0)$. We will use Assumptions~\ref{ass_str_cvx}, ~\ref{ass_smooth} and \ref{ass_Hess_lip} to bound the term $\frac{\|\hat{y}_k\|^2}{\hat{s}_k^\top \hat{y}_k}$ for any $k\geq 0$, as shown in Lemma~\ref{lem:y^2/sy} of Section~\ref{subsec:intermediate}. Moreover, the second term $\Psi(\hat{B}_0)$ depends on our choice of the initial Hessian approximation matrix $B_0$. Specifically, we will consider two different initializations: (i) $B_0 = L I$; (ii) $B_0 = \mu I$. As we shall discuss in the upcoming sections, these two choices result in different bounds and thus lead to a trade-off between the initial linear convergence rate and the final superlinear convergence rate. 

Having outlined our key propositions, Sections~\ref{sec:linear} and~\ref{sec:superlinear} will merge Proposition~\ref{lemma_bound} and Proposition~\ref{lemma_BFGS} to demonstrate that BFGS achieves global non-asymptotic linear and superlinear convergence rates, respectively. Our approach involves selecting an appropriate weight matrix $P$ and bounding the quantities in \eqref{eq:product} to derive the overall convergence rate. Specifically, we set $P = L I$ for global linear convergence and $P = \nabla^2 f(x_*)$ for superlinear convergence. The following section presents intermediate lemmas that will be used to establish these convergence bounds.

\subsection{Intermediate lemmas}\label{subsec:intermediate}

Next, we provide some intermediate results that lower bound the quantities $\hat{\alpha}_k$  and $\hat{q}_k$ defined in \eqref{eq:def_alpha_q_m} and the term $\frac{\|\hat{y}_k\|^2}{\hat{s}_k^\top \hat{y}_k}$ appearing in \eqref{eq:potential_decrease}. To do so, we first define the average Hessian matrices $J_k$ and $G_k$ as 
\begin{align}
    J_k &:= \int_{0}^{1}\nabla^2{f(x_k + \tau (x_{k + 1} - x_k))}d\tau, \label{def_J}\\
     G_k &:= \int_{0}^{1}\nabla^2{f(x_k + \tau (x_{*} - x_k))}d\tau.\label{def_G}
\end{align}
These two matrices play an important role in our analysis, since the fundamental theorem of calculus implies that $y_k = J_k s_k$ and $g_k = G_k (x_k-x^*)$ for any $k \geq 0$. We also define the weighted average Hessian matrix $\hat{J}_k = P^{-\frac{1}{2}}J_k P^{-\frac{1}{2}}$ for the given weight matrix $P \in  \mathbb{S}^d_{++}$. Moreover, we define a quantity $C_k$ that depends on the function value at the iterate $x_k$:
\begin{equation}\label{distance}
    C_k := \frac{2M}{\mu^{\frac{3}{2}}}\sqrt{2(f(x_k) - f(x_*))}, \qquad \forall k \geq 0,
\end{equation}
where $M$ is the Lipschitz constant of the Hessian in Assumption~\ref{ass_Hess_lip} and $\mu$ is the strong convexity parameter in Assumption~\ref{ass_str_cvx}. Given these definitions, in the following lemma, we characterize the relationship between different matrices that appear in our convergence analysis.

\begin{lemma}\label{lemma_Hessian}

Suppose Assumptions~\ref{ass_str_cvx}, \ref{ass_smooth}, and \ref{ass_Hess_lip} hold, and recall the definitions of the matrices $J_k$ in \eqref{def_J}, $G_k$ in \eqref{def_G}, and the quantity $C_k$ in \eqref{distance}. Then, the following statements hold: 
\begin{enumerate}[(a)]
    \item For any $k \geq 0$, we have that
    \begin{equation}\label{eq:J_k_vs_H*}
        \frac{1}{1 + C_k}\nabla^2{f(x_*)} \preceq J_k \preceq (1 + C_k)\nabla^2{f(x_*)}.
    \end{equation}
    \item For any $k\geq 0$, we have that
    \begin{equation}\label{eq:G_k_vs_H*}
        \frac{1}{1 + C_k}\nabla^2{f(x_*)} \preceq G_k \preceq (1 + C_k)\nabla^2{f(x_*)} .
    \end{equation}
    \item For any $k \geq 0$ and any $\hat{\tau} \in [0, 1]$, we have that
    \begin{equation}\label{eq:middle_hessian_vs_J_k}
        \frac{1}{1 + C_k}J_k \preceq \nabla^2{f(x_k + \hat{\tau} (x_{k + 1} - x_{k}))} \preceq (1 + C_k)J_k.
    \end{equation}
    \item For any $k\geq 0$ and $\tilde{\tau} \in [0, 1]$, we have that
    \begin{equation}\label{eq:middle_hessian_vs_G_k}
        \frac{1}{1 + C_k}G_k \preceq \nabla^2{f(x_k + \tilde{\tau} (x_{*} - x_{k}))} \preceq (1 + C_k)G_k.
    \end{equation}
\end{enumerate}
\end{lemma}

\begin{proof}
Please check Appendix~\ref{proof_of_lemma_Hessian}.
\end{proof}

After establishing Lemma~\ref{lemma_Hessian}, in the following three lemmas, we will provide bounds on the quantities $\hat{\alpha}_k$, $\hat{q}_k$ and $\frac{\|\hat{y}_k\|^2}{\hat{s}_k^\top \hat{y}_k}$, respectively. Note that $\hat{\alpha}_k$ is independent of the  choice of the weight matrix $P \in \mathbb{S}^d_{++}$, while $\hat{q}_k$ and $\frac{\|\hat{y}_k\|^2}{\hat{s}_k^\top \hat{y}_k}$ are determined by different options of the weight matrix $P$. Moreover, Lemmas~\ref{lemma_eigenvalue} and~\ref{lem:y^2/sy} are general results that hold for any sequence $\{x_k\}_{k\geq 0}$ and are not specifically tied to our update rule.

\begin{lemma}\label{lemma_kappa}
    Let $\{x_{k}\}_{k \geq 0}$ be the iterates generated by the BFGS algorithm with exact line search, and 
    {recall the definition $\hat{\alpha}_k = \frac{f(x_k) - f(x_{k+1})}{-\hat{g}_k^\top \hat{s}_k}$ in \eqref{eq:def_alpha_q_m}. Suppose Assumptions~\ref{ass_str_cvx}, \ref{ass_smooth}, and \ref{ass_Hess_lip} hold. Then, for any $k\geq 0$, we have}
    \begin{equation}\label{eq:alpha_lower_bound}
        \hat{\alpha}_k \geq \max\left\{\frac{1}{1 + \sqrt{\kappa}}, \frac{1}{2(1 + C_k)} \right\}.
    \end{equation}
\end{lemma}

\begin{proof}
    To begin with, note that we have $ -\hat{g}_k^\top \hat{s}_k = \hat{y}_k^\top \hat{s}_k$ due to Lemma~\ref{lemma_line_search}(b) and $\hat{y}_k^\top \hat{s}_k = {y}_k^\top{s}_k $ for any choice of the weight matrix $P$. Thus, $\hat{\alpha}_k$ can be equivalently defined as $\hat{\alpha}_k = \frac{f(x_k) - f(x_{k+1})}{{y}_k^\top {s}_k}$. 
    We first prove the first bound in \eqref{eq:alpha_lower_bound}.  By Assumptions~\ref{ass_str_cvx} and~\ref{ass_smooth}, the function $f$ is $\mu$-strongly convex and its gradient is $L$-Lipschitz. Then for any $x, y \in \mathbb{R}^d$, it holds that
    \begin{equation}\label{eq:interpolation}
    \begin{split}
         f(x) - f(y) - \nabla{f(y)}^\top (x - y) 
         & \geq \frac{\|\nabla{f(x)} - \nabla{f(y)}\|^2}{2(L - \mu)} + \frac{\mu L\|x - y\|^2}{2(L-\mu)} \\
         & \phantom{{}\geq{}} - \frac{\mu}{L - \mu}(\nabla{f(y)} - \nabla{f(x)})^\top (y - x) .
    \end{split}
    \end{equation}
    This is also known as the interpolation inequality; see, e.g., \cite[Theorem 4]{Taylor_convex}. By setting $x = x_k$, $y = x_{k + 1}$ in \eqref{eq:interpolation} and recalling that $s_k = x_{k+1} - x_k$, $y_k = \nabla f(x_{k+1}) - \nabla f(x_k)$ and $g_{k+1} = \nabla f(x_{k+1})$, we obtain that 
    \begin{equation*}
        f(x_{k}) - f(x_{k+1}) + g_{k+1}^\top s_k \geq \frac{1}{2(L-\mu)}\|y_k\|^2 + \frac{\mu L\|s_k\|^2}{2(L-\mu)} - \frac{\mu}{L-\mu} y_k^\top s_k.
    \end{equation*}
    Moreover, Lemma~\ref{lemma_line_search} shows that $g_{k+1}^\top s_k = 0$ due to exact line search. Thus, we can simplify the above inequality as
    \begin{align}\label{proof_lemma_kappa_1}
        f(x_{k}) - f(x_{k+1})  & \geq \frac{1}{2(L-\mu)}\|y_k\|^2 + \frac{\mu L\|s_k\|^2}{2(L-\mu)} - \frac{\mu}{L-\mu} y_k^\top s_k \nonumber\\
        & \geq \frac{\sqrt{\mu L}}{L-\mu}\|y_k\|\|s_k\| - \frac{\mu}{L-\mu} y_k^\top s_k \nonumber\\
    &\geq \left(\frac{\sqrt{\mu L}}{L-\mu} - \frac{\mu}{L-\mu}\right)y_k^\top s_k  = \frac{1}{1 + \sqrt{\kappa}} s_{k}^\top y_k,
    \end{align}
    where we used Young's inequality in the second inequality and the fact that $s_{k}^\top y_k \leq \|s_{k}\| \|y_k\|$ due to Cauchy-Schwartz inequality in the third inequality. Hence, we conclude that $\hat{\alpha}_k = \frac{f(x_k)-f(x_{k+1})}{s_k^\top y_k} \geq \frac{1}{1+\sqrt{\kappa}}$. 

    Now we proceed to establish the second lower bound on $\hat{\alpha}_k$. Given Taylor's theorem, there exists $\tau_k \in [0,1]$ such that 
    \begin{align*}
        f(x_k) & = f(x_{k + 1})  + g_{k + 1}^{\top}(x_k-x_{k+1}) \\
        & \quad + \frac{1}{2}(x_k-x_{k+1})^\top \nabla^2{f(x_k + \tau_k (x_{k + 1} - x_k))}(x_k-x_{k+1}) \\
        & = f(x_{k + 1}) + \frac{1}{2}s_{k}^\top \nabla^2{f(x_k + \tau_k (x_{k + 1} - x_k))} s_k,
    \end{align*}
    where we used $g_{k + 1}^{\top} s_k = 0$. Moreover, based on \eqref{eq:middle_hessian_vs_J_k} in Lemma~\ref{lemma_Hessian}(c), we have
    \begin{align*}
        s_{k}^\top \nabla^2{f(x_k + \tau_k (x_{k + 1} - x_k))} s_k \geq \frac{1}{1 + C_k}s_{k}^\top J_k s_k = \frac{1}{1 + C_k}s_{k}^\top y_k.
    \end{align*}
    Hence, we obtain that
    \begin{equation}\label{proof_lemma_kappa_2}
        f(x_k) - f(x_{k+1}) = \frac{1}{2}s_{k}^\top \nabla^2{f(x_k + \tau_k (x_{k + 1} - x_k))} s_k \geq \frac{1}{2(1 + C_k)} s_{k}^\top y_k.
    \end{equation}
    By combining the inequalities in \eqref{proof_lemma_kappa_1} and \eqref{proof_lemma_kappa_2}, the main claim follows. 
\end{proof}


\begin{lemma}\label{lemma_eigenvalue}
    Recall the definition $\hat{q}_k = \frac{\|\hat{g}_k\|^2}{f(x_k)-f(x_*)}$ in \eqref{eq:def_alpha_q_m}. Suppose Assumptions~\ref{ass_str_cvx},~\ref{ass_smooth}, and \ref{ass_Hess_lip} hold. Then we have the following results: 
    \begin{enumerate}[(a)]
        \item If we choose $P = LI$, then $\hat{q}_k \geq 2/\kappa$. 
        \item If we choose $P = \nabla^2 f(x_*)$, then $\hat{q}_k \geq 2/(1+C_k)^2$. 
    \end{enumerate}
\end{lemma}

\begin{proof}
    We first prove (a). When $P = LI$, we have $\hat{q}_k = \frac{\|{g}_k\|^2}{L(f(x_k)-f(x_*))}$. Since $f$ is $\mu$-strongly convex by Assumption~\ref{ass_str_cvx}, it holds that $\|\nabla{f(x_k)}\|^2 \geq 2\mu(f(x_k) - f(x_*)$ (see, e.g, \cite[Section 9.1.2]{boyd04}). Hence, we conclude that $\hat{q}_k \geq 2\mu/L = 2/\kappa$.
    
    Next, we prove (b). When $P = \nabla^2 f(x_*)$, we have $\|\hat{g}_k\|^2 = g_k^\top P^{-1} g_k = g_k^\top (\nabla^2 f(x_*))^{-1} g_k$. By applying Taylor's theorem with Lagrange remainder, there exists $\tilde{\tau}_k \in [0,1]$ such that 
    \begin{equation}\label{lemma_eigenvalue_proof_1}
        \begin{split}
            f(x_k)  
            & = f(x_*) + \nabla{f(x_*)}^\top (x_k - x_*) \\ 
            & \phantom{{}={}} + \frac{1}{2}(x_k - x_*)^\top \nabla^2{f(x_k + \tilde{\tau}_k(x_* - x_k))}(x_k - x_*),\\
            & = f(x_*) + \frac{1}{2}(x_k - x_*)^\top \nabla^2{f(x_k + \tilde{\tau}_k(x_* - x_k))}(x_k - x_*),
    \end{split}
    \end{equation}
    where we used the fact that $\nabla{f(x_*)} = 0$ in the last equality. Moreover, by the fundamental theorem of calculus, we have
    \begin{equation*}
        \nabla{f(x_k)} - \nabla{f(x_*)} = \int_{0}^{1}\nabla^2{f(x_k + \tau (x_{*} - x_k))} (x_k - x_*) \; d\tau = G_k(x_k - x^*),
    \end{equation*}
    where we use the definition of $G_k$ in \eqref{def_G}. Since $\nabla f(x_*) = 0$ and we denote $g_k = \nabla{f(x_k)}$, this further implies that
    \begin{equation}\label{lemma_eigenvalue_proof_2}
        x_k - x_* = G_k^{-1}(\nabla{f(x_k)} - \nabla{f(x_*)}) = G_k^{-1} g_k.
    \end{equation}
    Combining \eqref{lemma_eigenvalue_proof_1} and \eqref{lemma_eigenvalue_proof_2} leads to 
    \begin{equation}\label{eq:f_x_k-f_*}
        f(x_k) - f(x_*) = \frac{1}{2}g_k^\top G_k^{-1}\nabla^2{f(x_k + \tilde{\tau}_k(x_* - x_k))}G_k^{-1}g_k. 
    \end{equation}
    Based on \eqref{eq:middle_hessian_vs_G_k} in Lemma~\ref{lemma_Hessian}(d), we have $\nabla^2{f(x_k + \tilde{\tau}_k (x_{*} - x_{k}))} \preceq (1 + C_k)G_k$, which implies that
    \begin{equation}\label{lemma_eigenvalue_proof_3}
        G_k^{-1}\nabla^2{f(x_k + \tilde{\tau}_k (x_{*} - x_{k}))}G_k^{-1} \preceq (1 + C_k)G_k^{-1}.
    \end{equation}
    Moreover, it follows from \eqref{eq:G_k_vs_H*} in Lemma~\ref{lemma_Hessian}(b) that $\frac{1}{1 +  C_k} \nabla^2{f(x_*)} \preceq G_k$, which implies that
    \begin{equation}\label{lemma_eigenvalue_proof_4}
        G_k^{-1} \preceq (1 + C_k)(\nabla^2{f(x_*)})^{-1}.
    \end{equation}
    Combining \eqref{lemma_eigenvalue_proof_3} and \eqref{lemma_eigenvalue_proof_4}, we obtain that
    \begin{equation*}
        G_k^{-1}\nabla^2{f(x_k + \tilde{\tau}_k (x_{*} - x_{k}))}G_k^{-1} \preceq (1 + C_k)^2(\nabla^2{f(x_*)})^{-1},
    \end{equation*}
    and hence 
    \begin{equation*}
        g_k^\top G_k^{-1}\nabla^2{f(x_k + \tilde{\tau}_k (x_{*} - x_{k}))}G_k^{-1}g_k \leq (1 + C_k)^2 g_k^\top(\nabla^2{f(x_*)})^{-1} g_k.
    \end{equation*}
    By using \eqref{eq:f_x_k-f_*} and the fact that $ \|\hat{g}_k\|^2= g_k^\top (\nabla^2 f(x_*))^{-1} g_k$, we obtain
    \begin{equation*}
        \hat{q}_k = \frac{\|\hat{g}_k\|^2}{f(x_k)-f(x_*)} \geq \frac{2}{(1+C_k)^2},
    \end{equation*}
    and the claim follows.
\end{proof}

\begin{lemma}\label{lem:y^2/sy}
    Suppose Assumptions~\ref{ass_str_cvx},~\ref{ass_smooth}, and \ref{ass_Hess_lip} hold. Then we have 
    \begin{equation*}
        \frac{\|\hat{y}_k\|^2}{\hat{s}_k^\top \hat{y}_k} \leq \|\hat{J}_k\|, \qquad \forall k \geq 0. 
    \end{equation*}
    As a corollary, we have the following results: 
    \begin{enumerate}[(a)]
        \item If we choose $P = LI$, then $ \frac{\|\hat{y}_k\|^2}{\hat{s}_k^\top \hat{y}_k} \leq 1$. 
        \item If we choose $P = \nabla^2 f(x_*)$, then $\frac{\|\hat{y}_k\|^2}{\hat{s}_k^\top \hat{y}_k} \leq 1+C_k$. 
    \end{enumerate}
\end{lemma}

\begin{proof}
    Note that by the fundamental theorem of calculus, we have ${y}_k = {J}_k {s}_k$, which implies that $\hat{y}_k = \hat{J}_k \hat{s}_k$. Hence, we can bound 
    \begin{equation*}
    \frac{\|\hat{y}_k\|^2}{\hat{s}_k^\top \hat{y}_k} = \frac{\hat{s}_k^\top \hat{J}_k \hat{J}_k \hat{s}_k }{\hat{s}_k^\top \hat{J}_k \hat{s}_k} = \frac{\hat{s}_k^\top \hat{J}_k^{\frac{1}{2}} \hat{J}_k \hat{J}_k^{\frac{1}{2}} \hat{s}_k }{\| \hat{J}_k^{\frac{1}{2}} \hat{s}_k\|^2}  \leq \|\hat{J}_k\|.  
    \end{equation*}
    Hence, if $P = LI$, then $\|\hat{J}_k\| = \frac{1}{L}\|J_k\| \leq 1$ by Assumption~\ref{ass_smooth}, which proves the result in (a). Moreover, if $P = \nabla^2 f(x_*)$, then
    \begin{equation*}
    \|\hat{J}_k\| = \|(\nabla^2 f(x_*))^{-\frac{1}{2}}J_k(\nabla^2 f(x_*))^{-\frac{1}{2}}\| \leq 1 + C_k,
    \end{equation*}
    by \eqref{eq:J_k_vs_H*} in Lemma~\ref{lemma_Hessian}(a), which proves the result in (b).  
\end{proof}

\section{Global linear convergence rates}\label{sec:linear}

In this section, we establish the explicit global linear convergence rates for the BFGS method using an exact line search step size, marking one of the first non-asymptotic global linear convergence analyses of BFGS with a line search scheme. The subsequent global superlinear convergence analyses are established based on these linear rates. 

Specifically, we combine the fundamental inequality \eqref{eq:product} from Proposition~\ref{lemma_bound} with lower bounds of the terms $\hat{\alpha}_k$, $\hat{q}_k$, and $\cos^2(\hat{\theta}_k)/\hat{m}_k$ from Lemma~\ref{lemma_kappa}, \ref{lemma_eigenvalue}, \ref{lem:y^2/sy} and Proposition~\ref{lemma_BFGS} to prove all the global linear convergence rates. In this section, we set the weight matrix $P$ as $P = LI$ and we define the weighted matrix $\bar{B}_k$ as:
\begin{equation}\label{weighted_matrix_2}
    \bar{B}_k = \frac{1}{L}B_k, \qquad \text{ for}\  \ k\geq 0.
\end{equation}

In the following lemma, we prove the global linear convergence rate of the BFGS method for any choice of $B_0 \in \mathbb{S}^d_{++}$.

\begin{lemma}\label{lemma_6}
    Let $\{x_k\}_{k\geq 0}$ be the iterates generated by the BFGS method with exact line search and suppose that Assumptions~\ref{ass_str_cvx} and \ref{ass_smooth} hold. For any initial point $x_0 \in \mathbb{R}^{d}$ and any initial Hessian approximation matrix $B_0 \in \mathbb{S}^d_{++}$, we have the following global linear convergence rate for any $k \geq 1$,
    \begin{equation}\label{lemma_6_1}
        \frac{f(x_k) - f(x_*)}{f(x_0) - f(x_*)} \leq \left(1 - e^{-\frac{\Psi(\bar{B}_0)}{k}}\frac{2}{\kappa(1 + \sqrt{\kappa})}\right)^{k}.
    \end{equation}
\end{lemma}

\begin{proof}
    Our starting point is applying Proposition~\ref{lemma_bound} with the weight matrix $P$ chosen as $P = LI$. Specifically, \eqref{eq:product} shows that to obtain a convergence rate, it suffices to prove a lower bound on $\prod_{i=0}^{k-1} \frac{\hat{\alpha}_i \hat{q}_i}{\hat{m}_i} \cos^2(\hat{\theta}_i)$. It follows from Lemma~\ref{lemma_kappa} that $\hat{\alpha}_k = \frac{f(x_k) - f(x_{k + 1})}{s_k^\top y_k} \geq \frac{1}{\sqrt{\kappa} + 1}$ for any $k\geq 0$. Moreover, by applying Lemma~\ref{lemma_eigenvalue} with $P = L I$, we obtain that $\hat{q}_k = \frac{\|\hat{g}_k\|^2}{f(x_k)-f(x_*)} \geq \frac{2}{\kappa}$ for any $k \geq 0$. Furthermore, applying Proposition~\ref{lemma_BFGS} with $P = L I$, it follows from \eqref{eq:sum_of_logs} that 
    \begin{equation*}
        \sum_{i = 0}^{k - 1} \log{\frac{\cos^2(\hat{\theta}_i)}{\hat{m}_i}} \geq  - \Psi(\bar{B}_{0}) + \sum_{i = 0}^{k - 1}\left(1-\frac{\|\hat{y}_i\|^2}{\hat{s}_i^\top \hat{y}_i} \right) \geq - \Psi(\bar{B}_{0}),
    \end{equation*}
    where in the last inequality we used $\frac{\|\hat{y}_i\|^2}{\hat{s}_i^\top \hat{y}_i} \leq 1$ by Lemma~\ref{lem:y^2/sy} with $P = LI$. Taking exponentiation of both sides, this further implies that
    \begin{equation}\label{lemma_6_proof_3}
    \qquad \prod_{i = 0}^{k-1} \frac{\cos^2(\hat{\theta}_i)}{\hat{m}_i} \geq e^{-\Psi(\bar{B}_{0})}.
    \end{equation}
    Combining all the pieces above, we get 
    \begin{equation*}
        \prod_{i=0}^{k-1} \frac{\hat{\alpha}_i \hat{q}_i}{\hat{m}_i} \cos^2(\hat{\theta}_i) \geq \prod_{i=0}^{k-1} (\hat{\alpha}_i \hat{q}_i) \prod_{i=0}^{k-1} \frac{\cos^2(\hat{\theta}_i)}{\hat{m}_i} \geq \left(\frac{2}{\kappa (\sqrt{\kappa}+1)}\right)^k  e^{-\Psi(\bar{B}_{0})}.
    \end{equation*}
    Thus, it follows from Proposition~\ref{lemma_bound} that 
    \begin{equation*}
        \frac{f(x_{k}) - f(x_*)}{f(x_0) - f(x_*)} \leq \left[1 - \left(\prod_{i = 0}^{k-1} \frac{\hat{\alpha}_i \hat{q}_i}{\hat{m}_i} \cos^2(\hat{\theta}_i)\right)^{\frac{1}{k}}\right]^{k} \!\!\leq \left(1 - e^{-\frac{\Psi(\bar{B}_0)}{k}}\frac{2}{\kappa(1 + \sqrt{\kappa})}\right)^{k}. 
    \end{equation*}
    This completes the proof. 
\end{proof}

Notice that this result holds without the Hessian Lipschitz continuity assumption. In the next lemma, we present another version of the global linear convergence analysis with the additional assumption the Hessian of $f$ is $M$-Lipschitz. We show that the BFGS method with exact line search will eventually reach a global linear convergence rate of $(1 - {1}/\mathcal{O}{(\kappa)})^{k}$, which is the same as the gradient descent method.

\begin{lemma}\label{lemma_7}
    Let $\{x_k\}_{k\geq 0}$ be the iterates generated by the BFGS method with exact line search and suppose that Assumptions~\ref{ass_str_cvx}, \ref{ass_smooth} and \ref{ass_Hess_lip} hold. For any initial point $x_0 \in \mathbb{R}^{d}$ and any initial Hessian approximation matrix $B_0 \in \mathbb{S}^d_{++}$, we have the following global linear convergence rate for any $k \geq 1$,
    \begin{equation}\label{lemma_7_1}
        \frac{f(x_k) - f(x_*)}{f(x_0) - f(x_*)} \leq \left(1 - e^{-\frac{\Psi(\bar{B}_0)}{k}}\frac{1}{\kappa}\frac{1}{1 + C_0}\right)^{k}.
    \end{equation}
    Moreover, when $k \geq (1 + C_0)\Psi(\bar{B}_0) + 3C_0\kappa\min\{2(1 + C_0), (1 + \sqrt{\kappa})\}$, we have
    \begin{equation}\label{lemma_7_2}
        \frac{f(x_k) - f(x_*)}{f(x_0) - f(x_*)} \leq \left(1 - \frac{1}{3\kappa}\right)^{k}.
    \end{equation}
\end{lemma}

\begin{proof}
    We follow a similar argument as in the proof of Lemma~\ref{lemma_6} but with a different lower bound for $\hat{\alpha}_k$. Specifically, by Lemma~\ref{lemma_kappa}, we also have $\hat{\alpha}_k = \frac{f(x_k) - f(x_{k + 1})}{s_k^\top y_k} \geq \frac{1}{2(1 + C_k)}$. Combining this with $\hat{q}_k \geq 2/\kappa$ and \eqref{lemma_6_proof_3} leads to  
    \begin{equation}\label{eq:lower_product_C_i}
        \prod_{i=0}^{k-1} \frac{\hat{\alpha}_i \hat{q}_i}{\hat{m}_i} \cos^2(\hat{\theta}_i) \geq \prod_{i=0}^{k-1} (\hat{\alpha}_i \hat{q}_i) \prod_{i=0}^{k-1} \frac{\cos^2(\hat{\theta}_i)}{\hat{m}_i} \geq \left(\frac{1}{\kappa}\right)^k  e^{-\Psi(\bar{B}_{0})}\prod_{i = 0}^{k - 1}\frac{1}{1 + C_i}.
    \end{equation}
    To begin with, recall the definition that $C_i = \frac{M}{\mu^{\frac{3}{2}}}\sqrt{2(f(x_i) - f(x_*))}$. Since the objective function is non-increasing by Lemma~\ref{lemma_line_search}, it holds that $C_i \leq C_0$ for any $i \geq 0$. Thus, from \eqref{eq:lower_product_C_i} we have 
    \begin{equation*}
        \left(\prod_{i=0}^{k-1} \frac{\hat{\alpha}_i \hat{q}_i}{\hat{m}_i} \cos^2(\hat{\theta}_i) \right)^{\frac{1}{k}} \geq \frac{1}{\kappa}  e^{-\frac{\Psi(\bar{B}_{0})}{k}}\frac{1}{1 + C_0} .
    \end{equation*}
    Thus, by using Proposition~\ref{lemma_bound} we obtain \eqref{lemma_7_1}. 
    
    To prove the second claim in \eqref{lemma_7_2}, we use the fact that $1+x \leq e^x$ for any $x \in \reals$ to get
    \begin{equation}\label{eq:exp_lower_bound}
        \prod_{i = 0}^{k - 1}\frac{1}{1 + C_i} \geq \prod_{i = 0}^{k - 1} e^{-C_i} = e^{-\sum_{i=0}^{k-1} C_i}. 
    \end{equation}
    Combining \eqref{eq:lower_product_C_i} and \eqref{eq:exp_lower_bound} leads to 
    \begin{equation}\label{eq:bound_with_sum_C}
        \prod_{i=0}^{k-1} \frac{\hat{\alpha}_i \hat{q}_i}{\hat{m}_i} \cos^2(\hat{\theta}_i) \geq \left(\frac{1}{\kappa}\right)^k  e^{-\Psi(\bar{B}_{0})-\sum_{i=0}^{k-1} C_i}.
    \end{equation} 
    {Next, we prove an upper bound on $\sum_{i=0}^{k-1} C_i$. When $k \geq (1 + C_0)\Psi(\bar{B}_0) + 3C_0\kappa\min\{2(1 + C_0), (1 + \sqrt{\kappa})\}$, we have that $k \geq (1 + C_0)\Psi(\bar{B}_0)$, which implies that $k \geq \Psi(\bar{B}_0)$. Then \eqref{lemma_6_1} in Lemma~\ref{lemma_6} and \eqref{lemma_7_1} together imply that }
    \begin{equation*}
        \frac{f(x_k) - f(x_*)}{f(x_0) - f(x_*)} \leq \left(1 -\frac{1}{3\kappa}\max\left\{\frac{2}{1 + \sqrt{\kappa}}, \frac{1}{1 + C_0}\right\}\right)^{k},
    \end{equation*}
    where we used the fact that $e^{-\frac{\Psi(\bar{B}_0)}{k}} \geq e^{-1} \geq \frac{1}{3}$. Moreover, we decompose the sum $\sum_{i=0}^{k-1} C_i$ into two parts by $\sum_{i=0}^{k-1} C_i  = \sum_{i=0}^{\Psi(\bar{B}_0) - 1} C_i  + \sum_{i=\Psi(\bar{B}_0)}^{k-1} C_i$. For the first part, we have $\sum_{i=0}^{\Psi(\bar{B}_0) - 1} C_i \leq C_0 \Psi(\bar{B}_0)$. For the second part, by the definition of $C_i$, we have 
    \begin{align*}
        \sum_{i=\Psi(\bar{B}_0)}^{k-1} C_i & = C_0 \sum_{i=\Psi(\bar{B}_0)}^{k-1} \sqrt{\frac{f(x_i) - f(x_*)}{f(x_0) - f(x_*)}} \\
        &\leq C_0 \sum_{i = \Psi(\bar{B}_0)}^{k - 1}\left(1 - \frac{1}{3\kappa}\max\left\{\frac{2}{1 + \sqrt{\kappa}},\frac{1}{1 + C_0}\right\}\right)^{\frac{i}{2}} \\
        & \leq \frac{C_0}{1 - \sqrt{1 - \frac{1}{3\kappa}\max\{\frac{1}{1 + C_0}, \frac{2}{1 + \sqrt{\kappa}}\}}} \\
        & \leq 3C_0\kappa\min\{2(1 + C_0), {1 +\sqrt{\kappa}}\},
    \end{align*}
    where we used $\sqrt{1-x} \leq 1-\frac{1}{2}x$ for all $0\leq x \leq 1$ in the last inequality. Combining both inequalities, we arrive at 
    \begin{equation}\label{eq:bound_C}
        \sum_{i=0}^{k-1} C_i \leq C_0\Psi(\bar{B}_0) + 3C_0\kappa\min\{2(1 + C_0), 1 + \sqrt{\kappa}\}.
    \end{equation}
    Thus, when the number of iterations $k$ exceeds $(1 + C_0)\Psi(\bar{B}_0) + 3C_0\kappa\min\{2(1 + C_0), (1 + \sqrt{\kappa})\}$, by \eqref{eq:bound_with_sum_C} we have 
    \begin{equation*}
       \left( \prod_{i=0}^{k-1} \frac{\hat{\alpha}_i \hat{q}_i}{\hat{m}_i} \cos^2(\hat{\theta}_i) \right)^{\frac{1}{k}} \geq \frac{1}{\kappa}  e^{-\frac{1}{k}(\Psi(\bar{B}_{0})+\sum_{i=0}^{k-1} C_i)} \geq \frac{1}{e \kappa} \geq \frac{1}{3\kappa}. 
    \end{equation*} 
    Together with Proposition~\ref{lemma_bound}, this proves the second claim in \eqref{lemma_7_2}. 
\end{proof}

We summarize all the global linear convergence results from the above two lemmas in the following theorem.

\begin{theorem}\label{theorem_1}
    Let $\{x_k\}_{k\geq 0}$ be the iterates generated by the BFGS method with exact line search and suppose that Assumptions~\ref{ass_str_cvx}, \ref{ass_smooth} and \ref{ass_Hess_lip} hold. For any initial point $x_0 \in \mathbb{R}^{d}$ and any initial matrix $B_0 \in \mathbb{S}_{++}^{d}$, we have the following global linear convergence rate for any $k \geq 1$,
    \begin{equation}\label{theorem_1_1}
        \frac{f(x_k) - f(x_*)}{f(x_0) - f(x_*)} \leq \left(1 - e^{-\frac{\Psi(\bar{B}_0)}{k}}\frac{1}{\kappa}\max\left\{\frac{2}{1 + \sqrt{\kappa}}, \frac{1}{1 + C_0}\right\}\right)^{k},
    \end{equation}
    where $\bar{B}_0$ is defined in \eqref{weighted_matrix_2}. When $k \geq \Psi(\bar{B}_0)$, we have that
    \begin{equation}\label{theorem_1_2}
        \frac{f(x_k) - f(x_*)}{f(x_0) - f(x_*)} \leq \left(1 - \frac{1}{3\kappa}\max\left\{\frac{2}{1 + \sqrt{\kappa}}, \frac{1}{1 + C_0}\right\}\right)^{k}.
    \end{equation}
    Moreover, when $k \geq (1 + C_0)\Psi(\bar{B}_0) + 3C_0\kappa\min\{2(1 + C_0), 1 +\sqrt{\kappa}\}$, we have
    \begin{equation}\label{theorem_1_3}
        \frac{f(x_k) - f(x_*)}{f(x_0) - f(x_*)} \leq \left(1 - \frac{1}{3\kappa}\right)^{k}.
    \end{equation}
\end{theorem}

In Theorem~\ref{theorem_1}, we present three distinct linear convergence rates during different phases of the BFGS algorithm with exact line search. Specifically, the linear rate in \eqref{theorem_1_1} is applicable from the first iteration, but the contraction factor depends on the quantity $e^{-\Psi(\bar{B}_0)/k}$, which can be exponentially small and thus imply a slow convergence rate. However,  this quantity will be bounded away from zero as the number of iterations $k$ increases, resulting in an improved linear rate. In particular, for $k \geq \Psi(\bar{B}_0)$, the quantity $e^{-\Psi(\bar{B}_0)/k}$ is bounded below by $1/3$, leading to the second improved  linear convergence rate in \eqref{theorem_1_2}. Furthermore, as shown in Lemma~\ref{lemma_7}, after an additional $C_0\Psi(\bar{B}_0) + 3C_0\kappa\min\{2(1 + C_0), 1 +\sqrt{\kappa}\}$ iterations, we achieve the last linear convergence rate in \eqref{theorem_1_3}, which is comparable to that of gradient descent.


From the discussions above, we observe that the quantity $\Psi(\bar{B}_0)$ (recall that $\bar{B}_0 = \frac{1}{L} B_0$) plays a critical role in determining the transitions between different linear convergence phases, and a smaller $\Psi(\bar{B}_0)$ implies fewer iterations required to reach each linear convergence phase. Thus, we consider a special case to simplify our bounds, where $B_0 = \alpha I$ and $\alpha > 0$ is an arbitrary positive scalar. Given this simplification, we obtain $\Psi(\bar{B}_0) = \Psi(\frac{\alpha}{L}I) = \frac{\alpha}{L}d - d + d\log{\frac{L}{\alpha}}$. We apply this to Theorem~\ref{theorem_1} to establish the corresponding global linear rates, as stated in Corollary~\ref{corollary_1}.

\begin{corollary}\label{corollary_1}
    Let $\{x_k\}_{k\geq 0}$ be the iterates generated by the BFGS method with exact line search and suppose that Assumptions~\ref{ass_str_cvx}, \ref{ass_smooth} and \ref{ass_Hess_lip} hold. For any initial point $x_0 \in \mathbb{R}^{d}$ and the initial Hessian approximation matrix {$B_0 = \alpha I$ with $\alpha > 0$,} we have the following global convergence rate for any $k \geq 1$,
    \begin{equation}\label{corollary_1_1}
        \frac{f(x_k) - f(x_*)}{f(x_0) - f(x_*)} \leq \left(1 - e^{-\frac{\frac{\alpha}{L}d - d + d\log{\frac{L}{\alpha}}}{k}}\frac{1}{\kappa}\max\left\{\frac{2}{1 + \sqrt{\kappa}}, \frac{1}{1 + C_0}\right\}\right)^{k}.
    \end{equation}
    When $k \geq d(\frac{\alpha}{L} - 1 + \log{\frac{L}{\alpha}})$, we have that
    \begin{equation}\label{corollary_1_2}
        \frac{f(x_k) - f(x_*)}{f(x_0) - f(x_*)} \leq \left(1 - \frac{1}{3\kappa}\max\left\{\frac{2}{1 + \sqrt{\kappa}}, \frac{1}{1 + C_0}\right\}\right)^{k}.
    \end{equation}
    Moreover, when $k \geq (1 + C_0)d(\frac{\alpha}{L} - 1 + \log{\frac{L}{\alpha}}) + 3C_0\kappa\min\{2(1 + C_0), 1 +\sqrt{\kappa}\}$, we have that
    \begin{equation}\label{corollary_1_3}
        \frac{f(x_k) - f(x_*)}{f(x_0) - f(x_*)} \leq \left(1 - \frac{1}{3\kappa}\right)^{k}.
    \end{equation}
\end{corollary}


The above corollary characterizes the behavior of BFGS with exact line search when the initial Hessian approximation is a scaled identity matrix. In the following paragraphs, we refine these results by analyzing the bounds for specific values of $\alpha$.

First, we examine two extreme cases for initialization: $\alpha = L$ (where $ B_0 = L I $) and $\alpha = \mu$ (where $B_0 = \mu I $). The former corresponds to an upper bound on the eigenvalues of the Hessian, while the latter uses a lower bound.
Note that in the case where $B_0 = LI$,  we have $d(\frac{\alpha}{L} - 1 + \log{\frac{L}{\alpha}}) = d(\frac{L}{L} - 1 + \log{\frac{L}{L}}) = 0$ and we achieve the global linear convergence rate in \eqref{corollary_1_2} for any $k \geq 1$. Moreover, when $k \geq 3C_0\kappa\min\{2(1 + C_0), (1 + \sqrt{\kappa})\}$, we reach the second linear convergence rate in \eqref{corollary_1_3}.
In the case where $B_0 = \mu I$, we have $d(\frac{\alpha}{L} - 1 + \log{\frac{L}{\alpha}}) = d(\frac{1}{\kappa} - 1 + \log{\kappa}) \leq d\log \kappa$. We have the following global convergence rate for any $k \geq 1$,
\begin{equation}
    \frac{f(x_k) - f(x_*)}{f(x_0) - f(x_*)} \leq \left(1 - e^{-\frac{d\log{\kappa}}{k}}\frac{1}{\kappa}\max\left\{\frac{2}{1 + \sqrt{\kappa}}, \frac{1}{1 + C_0}\right\}\right)^{k}.
\end{equation}
When $k \geq d\log{\kappa}$, we achieve the global linear convergence rate in \eqref{corollary_1_2}. Moreover, when $k \geq (1 + C_0)d\log{\kappa} + 3C_0\kappa\min\{2(1 + C_0), 1 +\sqrt{\kappa}\}$, we reach the second linear convergence rate in \eqref{corollary_1_3}.
Comparing the above results, we observe that BFGS with $B_0 = \mu I$ requires additional $d\log \kappa$ iterations to achieve a similar linear rate as in the first case. However, as we present in the next section, the choice of the initial Hessian approximation matrix $B_0 = \mu I$ could achieve a superlinear rate faster. This trade-off between the linear and superlinear convergence phase is the fundamental consequence of different choices of the initial Hessian approximation matrix in our convergence analysis.


While the special cases discussed above are valuable for theoretical comparison, they may not be practical for selecting the initial Hessian approximation, as the constants \(\mu\) and \(L\) are often unknown. A more practical choice, which can be easily computed, is to set \( B_0 = cI \), where \( c \) is determined based on gradient and variable differences between two randomly selected points. Specifically, \( c \) is given by $c = \frac{s^\top y}{\|s\|^2}$ where \( s = x_2 - x_1 \) and \( y = \nabla f(x_2) - \nabla f(x_1) \), with \( x_1 \) and \( x_2 \) being two randomly chosen vectors. This choice ensures that \( c \in [\mu, L] \). For this initialization of \( B_0 \), we can establish the bound $d\left(\frac{c}{L} - 1 + \log{\frac{L}{c}}\right) \leq d\log\kappa$. Applying this upper bound to our linear convergence result in Corollary~\ref{corollary_1}, we obtain that when $k \geq d\log\kappa$, we have the linear convergence rate in \eqref{corollary_1_2}. When $k \geq (1 + C_0)d\log\kappa + 3C_0\kappa\min\{2(1 + C_0), 1 +\sqrt{\kappa}\}$, we have the linear rate in \eqref{corollary_1_3}.

\section{Global superlinear convergence rates}\label{sec:superlinear}

In this section, we establish the non-asymptotic global superlinear convergence rate of BFGS with exact line search, employing a similar approach to the global linear convergence rate analysis from the previous section. We utilize the framework from Proposition~\ref{lemma_bound} and integrate the lower bounds from Lemmas~\ref{lemma_kappa}, \ref{lemma_eigenvalue}, \ref{lem:y^2/sy}, and Proposition~\ref{lemma_BFGS}. The key distinction lies in the choice of the weight matrix: instead of $P = LI$ used in the linear convergence analysis, we opt for $P = \nabla^2{f(x_*)}$ for the global superlinear convergence proof.

We define the weighted matrix $\tilde{B}_k$ as:
\begin{equation}\label{weighted_matrix_3}
    \tilde{B}_k = \nabla^2 f(x_*)^{-\frac{1}{2}} B_k \nabla^2 f(x_*)^{-\frac{1}{2}}, \qquad \text{ for}\  \ k\geq 0.
\end{equation}
In the following proposition, we first provide a general global convergence bound with an arbitrary initial Hessian approximation matrix $B_0 \in \mathbb{S}^d_{++}$. All the global superlinear convergence rates are based on the following proposition.

\begin{proposition}\label{proposition_3}
    Let $\{x_k\}_{k\geq 0}$ be the iterates generated by the BFGS method with exact line search and suppose that Assumptions~\ref{ass_str_cvx}, \ref{ass_smooth} and \ref{ass_Hess_lip} hold. Recall the definition of $C_k$ in \eqref{distance} and $\Psi(\cdot)$ in \eqref{potential_function}.
    For any initial point $x_0 \in \mathbb{R}^{d}$ and any initial Hessian approximation matrix $B_0 \in \mathbb{S}^d_{++}$, the following result holds for any $k \geq 1$,
    \begin{equation}\label{proposition_3_1}
        \frac{f(x_k) - f(x_*)}{f(x_0) - f(x_*)} \leq \left(\frac{\Psi(\Tilde{B}_0) + 4\sum_{i = 0}^{k - 1}C_i}{k}\right)^{k}.
    \end{equation}
\end{proposition}

\begin{proof}
    Recall that we choose the weight matrix as $P = \nabla^2 f(x_*)$ throughout the proof. From Lemma~\ref{lemma_kappa} and Lemma~\ref{lemma_eigenvalue}(b), we have $\hat{\alpha}_k \geq \frac{1}{2(1 + C_k)}$ and $\hat{q}_k \geq \frac{2}{(1 + C_k)^2}$. Hence, using the inequality $1+x \leq e^x$ for any $x\geq 0$, it follows that 
    \begin{equation}\label{eq:product_alpha_q}
        \prod_{i = 0}^{k - 1}(\hat{\alpha}_i \hat{q}_i) \geq \prod_{i = 0}^{k - 1} \frac{1}{(1+C_k)^3} \geq \prod_{i = 0}^{k - 1} e^{-3C_k} = e^{-3\sum_{i = 0}^{k - 1}C_i}.
    \end{equation}
    Moreover, by using the inequality \eqref{eq:sum_of_logs} in Proposition~\ref{lemma_BFGS} with $P = \nabla^2{f(x_*)}$, we obtain that
    \begin{equation*}
        \sum_{i = 0}^{k - 1} \log{\frac{\cos^2(\hat{\theta}_i)}{\hat{m}_i}} \geq  - \Psi(\Tilde{B}_{0}) + \sum_{i = 0}^{k - 1}\left(1-\frac{\|\hat{y}_i\|^2}{\hat{s}_i^\top \hat{y}_i} \right) \geq - \Psi(\Tilde{B}_{0}) - \sum_{i = 0}^{k - 1}C_i,
    \end{equation*}
    where in the last inequality we used the fact that $\frac{\|\hat{y}_i\|^2}{\hat{s}_i^\top \hat{y}_i} \leq 1 + C_i$ from Lemma~\ref{lem:y^2/sy}(b). This further implies that
    \begin{equation}\label{proposition_3_proof_3}
        \prod_{i = 0}^{k - 1} \frac{\cos^2(\hat{\theta}_i)}{\hat{m}_i} \geq e^{- \Psi(\Tilde{B}_{0}) - \sum_{i = 0}^{k - 1}C_i}.
    \end{equation}
    Combining \eqref{eq:product_alpha_q}, \eqref{proposition_3_proof_3}, and \eqref{eq:product} from Proposition~\ref{lemma_bound}, we prove that
    \begin{align*}
        \frac{f(x_{k}) - f(x_*)}{f(x_0) - f(x_*)} & \leq \left[1 - \left(\prod_{i = 0}^{k-1} \frac{\hat{\alpha}_i \hat{q}_i}{\hat{m}_i} \cos^2(\hat{\theta}_i)\right)^{\frac{1}{k}}\right]^{k} \\
        & \leq \left[1 - \left( e^{-3\sum_{i = 0}^{k - 1}C_i} e^{- \Psi(\Tilde{B}_{0}) - \sum_{i = 0}^{k - 1}C_i}\right)^{\frac{1}{k}}\right]^{k} \\
        & = \left(1 - e^{- \frac{\Psi(\Tilde{B}_{0}) + 4\sum_{i = 0}^{k - 1}C_i}{k}} \right)^{k} \leq \left(\frac{\Psi(\Tilde{B}_0) + 4\sum_{i = 0}^{k - 1}C_i}{k}\right)^{k},
    \end{align*}
    where the last inequality is due to the fact that $1 - e^{-x} \leq x$ for any $x$.
\end{proof}

The above global result shows that the error after $k$ iterations for the BFGS update with exact line search depends on the potential function of the weighted initial Hessian approximation matrix $\Tilde{B}_0$, i.e., $\Psi(\Tilde{B}_0)$, and the sum of weighted function value optimality gap, i.e., $\sum_{i = 0}^{k - 1}C_i$. This result forms the foundation of our superlinear result, since if we can demonstrate that the sum $\sum_{i = 0}^{k - 1}C_i$ is bounded above, it leads to a superlinear rate of the form $\mathcal{O}((1/k)^k)$. 

Having established the non-asymptotic global linear convergence rate of BFGS in the previous section, we can leverage it to show that the sum $\sum_{i = 0}^{k - 1}C_i$ is uniformly bounded above, allowing us to establish an explicit upper bound for this finite sum. In the following theorem, we apply the linear convergence results from Section~\ref{sec:linear} to prove the non-asymptotic global superlinear convergence rates of BFGS with exact line search for any initial Hessian approximation matrix $B_0 \in \mathbb{S}^d_{++}$.

\begin{theorem}\label{theorem_2}
    Let $\{x_k\}_{k\geq 0}$ be the iterates generated by the BFGS method with exact line search and suppose that Assumptions~\ref{ass_str_cvx}, \ref{ass_smooth} and \ref{ass_Hess_lip} hold. For any initial point $x_0 \in \mathbb{R}^{d}$ and any initial Hessian approximation matrix $B_0 \in \mathbb{S}^d_{++}$, we have the following superlinear convergence rate,
    \begin{equation}\label{theorem_2_1}
        \frac{f(x_k) - f(x_*)}{f(x_0) - f(x_*)} \leq \left(\frac{\Psi(\Tilde{B}_{0}) + 4C_0\Psi(\bar{B}_0) + 12C_0 \kappa\min\{2(1 + C_0), 1 +\sqrt{\kappa}\}}{k}\right)^{k},
    \end{equation}
    where $\bar{B}_0$ and $\tilde{B}_0$ are defined in \eqref{weighted_matrix_2} and \eqref{weighted_matrix_3}.
\end{theorem}

\begin{proof}
    From \eqref{eq:bound_C} in Lemma~\ref{lemma_7}, we know that for $k \geq 1$,
    \begin{equation}\label{theorem_2_proof_1}
        \sum_{i = 0}^{k - 1}C_i \leq C_0 \Psi(\bar{B}_0) + 3C_0\kappa\min\{2(1 + C_0), 1 +\sqrt{\kappa}\}).
    \end{equation}
    Leveraging \eqref{theorem_2_proof_1} and \eqref{proposition_3_1} in Lemma~\ref{proposition_3}, we prove that for $k \geq 1$,
    \begin{equation*}
    \begin{split}
        \frac{f(x_k) - f(x_*)}{f(x_0) - f(x_*)} & \leq \left(\frac{\Psi(\Tilde{B}_0) + 4\sum_{i = 0}^{k - 1}C_i}{k}\right)^{k} \\ 
        & \leq \left(\frac{\Psi(\Tilde{B}_{0}) + 4C_0\Psi(\bar{B}_0) + 12C_0 \kappa\min\{2(1 + C_0), 1 +\sqrt{\kappa}\}}{k}\right)^{k},
    \end{split}
    \end{equation*}
    and the proof is complete. 
\end{proof}

This result indicates that BFGS with exact line search achieves a superlinear convergence rate when the number of iterations satisfies the condition $k\geq \Psi(\Tilde{B}_{0}) + 4C_0\Psi(\bar{B}_0) + 12C_0 \kappa\min\{2(1 + C_0), 1 +\sqrt{\kappa}\}$. The initial matrix $B_0$ critically influences the required iterations to attain this rate, as it appears in the numerator of the upper bound through $\Tilde{B}_{0}= \nabla^2 f(x_*)^{-\frac{1}{2}} B_0 \nabla^2 f(x_*)^{-\frac{1}{2}}$ and $\bar{B}_0=(1/L)B_0$. Thus, different choices of $B_0$ yield different values for $\Psi(\Tilde{B}_{0}) + 4C_0\Psi(\bar{B}_0)$, affecting the number of iterations required for superlinear convergence. Indeed, one can try to optimize the choice of $B_0$ to make the expression $\Psi(\Tilde{B}_{0}) + 4C_0\Psi(\bar{B}_0)$ as small as possible. 

Now we consider the special case where $B_0 = \alpha I$ with \( \alpha > 0 \) as any positive constant. For this case, we have $\Psi(\bar{B}_0) = \Psi(\frac{\alpha}{L}I) = \frac{\alpha}{L}d - d + d\log{\frac{L}{\alpha}}$ and $\Psi(\tilde{B}_0) = \Psi(\alpha \nabla^2{f(x_*)^{-1}}) = \alpha\mathbf{Tr}(\nabla^2{f(x_*)^{-1}}) - d - \log{\mathbf{Det}(\alpha\nabla^2{f(x_*)^{-1}})} \leq d(\frac{\alpha}{\mu} - 1 + \log{\frac{L}{\alpha}})$ from Assumptions~\ref{ass_str_cvx} and \ref{ass_smooth}. 
Applying these bounds in Theorem~\ref{theorem_2}, we obtain the following corollary.

\begin{corollary}\label{corollary_2}
    Let $\{x_k\}_{k\geq 0}$ be the iterates generated by the BFGS method with exact line search and suppose that Assumptions~\ref{ass_str_cvx}, \ref{ass_smooth} and \ref{ass_Hess_lip} hold. For any initial point $x_0 \in \mathbb{R}^{d}$ and the initial Hessian approximation matrix $B_0 = \alpha I$ with $\alpha > 0$, we have the following superlinear convergence rate,
    \begin{equation}
    \begin{split}
        & \frac{f(x_k) - f(x_*)}{f(x_0) - f(x_*)} \leq \\
        & \left(\frac{d(\frac{\alpha}{\mu} - 1 + \log{\frac{L}{\alpha}}) + 4C_0d(\frac{\alpha}{L} - 1 + \log{\frac{L}{\alpha}}) + 12C_0 \kappa\min\{2(1 + C_0), 1 +\sqrt{\kappa}\}}{k}\right)^{k}.
    \end{split}
    \end{equation}
\end{corollary}

The above corollary characterizes the non-asymptotic superlinear convergence rate of BFGS with exact line search when the initial Hessian approximation is an identity matrix multiplied by a constant $\alpha$. Similar to the linear convergence analysis, we present the superlinear convergence rates for specific values of $\alpha$ in the following paragraphs.

When $\alpha = L$ ($B_0 = L I$), we have $\Psi(\bar{B}_0) = d(\frac{\alpha}{L} - 1 + \log{\frac{L}{\alpha}}) = 0$ and $\Psi(\Tilde{B}_{0}) = d(\frac{\alpha}{\mu} - 1 + \log{\frac{L}{\alpha}}) \leq d\kappa$. Hence, we obtain the superlinear convergence rate
\begin{equation*}
    \frac{f(x_k) - f(x_*)}{f(x_0) - f(x_*)} \leq \left(\frac{d\kappa + 12C_0\kappa\min\{2(1 + C_0), (1 + \sqrt{\kappa})\}}{k}\right)^{k}.
\end{equation*}
Similarly, when $\alpha = \mu$ ($B_0 = \mu I$), we have $\Psi(\bar{B}_{0}) = d(\frac{\alpha}{L} - 1 + \log{\frac{L}{\alpha}}) \leq d\log\kappa$ and $\Psi(\Tilde{B}_{0}) = d(\frac{\alpha}{\mu} - 1 + \log{\frac{L}{\alpha}}) \leq d\log\kappa$. This leads to the superlinear convergence rate
\begin{equation*}
    \frac{f(x_k) - f(x_*)}{f(x_0) - f(x_*)} \leq \left(\frac{(1 + 4C_0)d\log{\kappa} + 12C_0\kappa\min\{2(1 + C_0), 1 +\sqrt{\kappa}\}}{k}\right)^{k}.
\end{equation*}
As shown in the above two results, choosing $B_0 = L I$ minimizes $\Psi(\bar{B}_0)$, resulting in $\Psi(\bar{B}_0)=0$. However, $\Psi(\Tilde{B}_{0})$ in this case could be as large as $d\kappa$. On the other hand, setting $B_0 = \mu I$ yields a more favorable upper bound, ensuring that both $\Psi(\bar{B}_0)$ and $\Psi(\Tilde{B}_{0})$ are bounded by $d\log \kappa$.
Hence, initializing the Hessian approximation with $B_0 = \mu I$ instead of $B_0 = LI$ could result in fewer iterations to reach the superlinear convergence phase. 
Generally, during the initial linear convergence stage, the iterates generated by the BFGS method with $B_0 = L I$ outperform those with $B_0 = \mu I$, due to a faster linear convergence speed. However, the BFGS method with $B_0 = \mu I$ transitions to the ultimate superlinear convergence phase in fewer iterations compared to $B_0 = L I$. This phenomenon has also been observed in our experiments in Section~\ref{sec:experiments}. 

As in the linear convergence analysis, we also consider the practical initial Hessian approximation: \( B_0 = cI \), where \( c \) is $\frac{s^\top y}{\|s\|^2}$, with $s = x_2 - x_1$, $y = \nabla f(x_2) - \nabla f(x_1)$, and $x_1, x_2$ as two random vectors. For this choice of $B_0$, we can derive the following upper bounds: \(\Psi(\bar{B}_0) \leq d\left(\frac{c}{L} - 1 + \log{\frac{L}{c}}\right) \leq d\log\kappa\) and \(\Psi(\tilde{B}_0) \leq d\left(\frac{c}{\mu} - 1 + \log{\frac{L}{c}}\right) \leq 2d\kappa\). Applying these values of $\Psi(\bar{B}_0)$ and $\Psi(\tilde{B}_0)$ to our superlinear convergence result in Corollary~\ref{corollary_2}, we can obtain the following convergence guarantees for $B_0 = c I$: 
\begin{equation}
    \frac{f(x_k) - f(x_*)}{f(x_0) - f(x_*)} \leq \left(\frac{2d\kappa + 4C_0d\log\kappa + 12C_0 \kappa\min\{2(1 + C_0), 1 +\sqrt{\kappa}\}}{k}\right)^{k}.
\end{equation}

While all of our presented results are global and do not impose any initial condition on $x_0$, in the following remark, we present a potential local result when $B_0 = \mu I$.

\begin{remark} 
Consider the scenario where BFGS starts at a point $x_0$ near the optimal solution $x_*$ such that the initial error condition $C_0 = \mathcal{O}( {1}/{\sqrt{\kappa}})$ is satisfied, i.e., $f(x_0)-f(x_*) = \mathcal{O}( \frac{\mu^4}{M^2{L }})$. In this case, we can establish that $(1 + 4C_0)d\log{\kappa} = \mathcal{O}(d\log{\kappa})$ and $C_0 \kappa\min\{1 + C_0, \sqrt{\kappa}\} = \mathcal{O}(1)$. Thus, when $B_0 = \mu I$, we obtain the local superlinear convergence rate of $\mathcal{O}(\frac{d\log{\kappa}}{k})^{k}$, which aligns with the local convergence result in \cite{rodomanov2020ratesnew}. It is noteworthy that the local result in \cite{rodomanov2020ratesnew} relied on a unit step size, while our local side result is derived using exact line search.
\end{remark}

\section{Discussions}\label{sec:discussions}

\noindent\textbf{Comparison with local non-asymptotic analysis.} In this section, we discuss the recent non-asymptotic local convergence results for BFGS and DFP in \cite{rodomanov2020rates,rodomanov2020ratesnew,qiujiang2020quasinewton} and explain why these results cannot be easily extended to achieve global complexity bounds. 

To begin with, note that these results are crucially based on local analysis and only apply when the iterates are close to the optimal solution $x_*$ and the step size $\eta_k$ is set to 1 in this local region. Therefore, to extend their results into a global convergence guarantee, one plausible strategy is to employ a line search scheme to ensure global convergence, and then switch to the local analysis when the iterates enter the region of local convergence. However, this approach faces several challenges. 
 
First, it remains unclear how to explicitly upper bound the number of iterations until the line search subroutine accepts the unit step size $\eta_k = 1$. Moreover, assume that the iterates enter the region of local convergence after $k_0$ iterations and we have $\eta_k = 1$ for all $k \geq k_0$. Even then, there is no guarantee that the Hessian approximation matrix $B_{k_0}$ will satisfy the necessary conditions required for the local analysis in \cite{rodomanov2020rates,rodomanov2020ratesnew,qiujiang2020quasinewton}. Specifically, for the analysis in \cite{qiujiang2020quasinewton} to hold, $B_{k_0}$ must be sufficiently close to the exact Hessian matrix, which is not satisfied in general. Regarding \cite{rodomanov2020ratesnew,rodomanov2020rates}, we note that their analyses depend on the condition number of $B_{k_0}$, which could be exponentially large and thus render the superlinear rate meaningless. To be more concrete, inspecting the proofs in \cite[Lemma 5.4]{rodomanov2020ratesnew} and \cite[Theorem 4.2]{rodomanov2020rates} reveals that the superlinear convergence rate occurs when $k = \Omega(\Psi(\check{B}_{k_0}^{-1}))$ and $k = \Omega(\Psi(\check{B}_{k_0}))$, respectively, where $\check{B}_{k_0} = J_{k_0}^{-{1}/{2}}B_{k_0}J_{k_0}^{-{1}/{2}}$ with $J_{k_0}$ defined in  \eqref{def_J} and $\Psi(\cdot)$ is the potential function defined in \eqref{potential_function}.  Consequently, it is essential to establish bounds for the smallest and largest eigenvalues of $\check{B}_{k_0}$. However, the current theory indicates (see e.g. \cite[Theorem 4.1]{rodomanov2020rates}) that $ e^{-2\kappa M \lambda_0} I \preceq \check{B}_{k_0} \preceq e^{2\kappa M \lambda_0} I$, where $\lambda_0 = \|(\nabla^2 f(x_0))^{-\frac{1}{2}}\nabla f(x_0)\|$ denotes the initial Newton decrement. This suggests that without a sufficiently small $\lambda_0$, the extreme eigenvalues of $\check{B}_{k_0}$ will be exponentially dependent on the condition number $\kappa$, leading to $\Psi(\check{B}_{k_0}^{-1}), \Psi(\check{B}_{k_0}) = \Omega(d e^{2\kappa M \lambda_0})$. Hence, a superlinear rate will be achieved only after $\Omega(d e^{2\kappa M \lambda_0})$ iterations.

Our convergence framework also diverges significantly from the previous works \cite{rodomanov2020rates,rodomanov2020ratesnew,qiujiang2020quasinewton} in terms of the proof strategy. Specifically, the approach in the aforementioned studies employs an induction argument to control the largest and smallest eigenvalues of the Hessian approximation matrix $B_k$ and prove a local linear convergence rate. In comparison, as presented in Sections~\ref{sec:linear} and \ref{sec:superlinear}, we prove global linear and superlinear convergence rates without explicitly establishing upper or lower bounds on the eigenvalues of $B_k$. This marks a notable departure from the local convergence analysis in \cite{rodomanov2020rates}, \cite{rodomanov2020ratesnew}, and \cite{qiujiang2020quasinewton}. 

\vspace{2mm}

\noindent\textbf{Comparison with global asymptotic analysis.} As mentioned in Section~\ref{sec:basic}, our convergence analysis framework resembles the approach taken in~\cite{Powell,byrd1987global,QN_tool} for proving asymptotic linear convergence rates of classical quasi-Newton methods such as BFGS and DFP. While these works considered inexact line search schemes and thus are different from our exact line search setting, they used a similar inequality as \eqref{eq:product} in Proposition~\ref{lemma_bound} to express the convergence rate in terms of the angle $\hat{\theta}_k$. Moreover, the authors in \cite{Powell} and \cite{byrd1987global} analyzed the traces and the determinants of the Hessian approximation matrices $\{B_k\}_{k \geq 0}$ separately to lower bound $\prod_{i = 0}^{k - 1}\cos{(\hat{\theta}_i)}$. Later, this process was simplified in \cite{QN_tool} by introducing the potential function $\Psi(\cdot)$ given in \eqref{potential_function}, combining the trace and determinant together as in our Proposition~\ref{lemma_BFGS}. However, since their main focus is on asymptotic convergence, we note that these previous works only demonstrate that $(\prod_{i = 0}^{k - 1}\cos{(\hat{\theta}_i)})^{{1}/{k}}$ is lower bounded by a constant, without giving an explicit form. Furthermore, our work builds upon previous analyses by incorporating a weight matrix $P$, while earlier works correspond to setting $P = I$. Another notable difference is that we keep the term $\hat{m}_k$ and lower bound the term $\cos^2(\hat{\theta}_k)/\hat{m}_k$ as shown in Proposition~\ref{lemma_BFGS}, whereas previous works relied on a looser bound for $\hat{m}_k$.  These refinements enable us to provide a tighter linear convergence rate for the BFGS method. 

On the other hand, in demonstrating superlinear convergence, our approach deviates significantly from that of \cite{Powell,byrd1987global,QN_tool}. Specifically, the previous works relied on the Dennis-Mor\'e condition, i.e., $\lim_{k \to \infty}\frac{\|(B_k - \nabla^2{f(x_*)})s_k\|}{\|s_k\|} = 0$, to establish asymptotic superlinear convergence. In comparison, we use the same framework outlined in Section~\ref{sec:basic} to establish both linear and superlinear convergence rates. The key distinction lies in the choice of the weight matrix $P$: we choose $P = L I$ for showing linear convergence and $P = \nabla^2 f(x_*)$ for showing superlinear convergence. Thus, we provide a unified framework for studying the global non-asymptotic convergence of BFGS. 

\section{Numerical experiments}\label{sec:experiments}

In this section, we present our numerical experiments to corroborate our convergence rate guarantees, and in particular, we explore the difference between the convergence paths of BFGS under different initializations of $B_0$. We further compare these variants of BFGS implementations with the gradient descent algorithm when deployed with exact line search. In our numerical experiments, all the step sizes used in BFGS with different $B_0$ and gradient descent are computed by the exact line search condition defined in \eqref{exact_line_search}. Specifically, we use MATLAB's ``fminsearch'' function from its optimization toolbox to determine the exact line search step size for all algorithms. In our experiments, all initial points are chosen as random vectors in the corresponding Euclidean vector spaces. 

We focus on a hard cubic objective function defined in \cite[Section 5]{hard_cubic}, i.e.,
\begin{equation}
    f(x) = \frac{\alpha}{12}\left(\sum_{i = 1}^{d - 1}g(v_i^\top x - v_{i + 1}^\top x) - \beta v_1^\top x\right) + \frac{\lambda}{2}\|x\|^2,
\end{equation}
and $g: \mathbb{R} \to \mathbb{R}$ is defined as
\begin{equation}
    g(w) =
    \begin{cases}
        \frac{1}{3}|w|^3 & |w| \leq \Delta, \\
        \Delta w^2 - \Delta^2 |w| + \frac{1}{3}\Delta^3  & |w| > \Delta,
    \end{cases}       
\end{equation}
where $\alpha, \beta, \lambda, \Delta \in \mathbb{R}$ are hyper-parameters and $\{v_i\}_{i = 1}^{n}$ are standard orthogonal unit vectors in $\mathbb{R}^{d}$. This hard cubic function is used to establish a lower bound for second-order methods. 

In Figure 1, we compare the gradient descent method and BFGS with different initialization of $B_0$: $B_0 = LI$, $B_0 = \mu I$, $B_0 = 10LI$, $B_0 = 0.1\mu I$, $B_0 = \sqrt{L\mu}I$ and $B_0 = cI$ where $c = \frac{s^\top y}{\|s\|^2}$. 
Here, $s = x_2 - x_1$, $y = \nabla f(x_2) - \nabla f(x_1)$, and $x_1, x_2$ as two randomly selected vectors. Note that $c \in [\mu, L]$ and the choice of $B_0 = cI$ is the most commonly used initial Hessian approximation matrix in practice \cite{nocedal2006numerical}. In (a), (b), and (c) of Figure~\ref{fig:1},  
we vary the problem's dimension while keeping the condition number as 1,000. Conversely, in (d), (e), and (f),  
we fix the problem's dimension as 600 and vary the condition number.

\begin{figure}
    \centering
    \subfigure[$d=50$, $\kappa=10^3$, $L=1$, $\mu=10^{-3}$, $c=10^{-1}$.]{\includegraphics[width=0.32\linewidth]{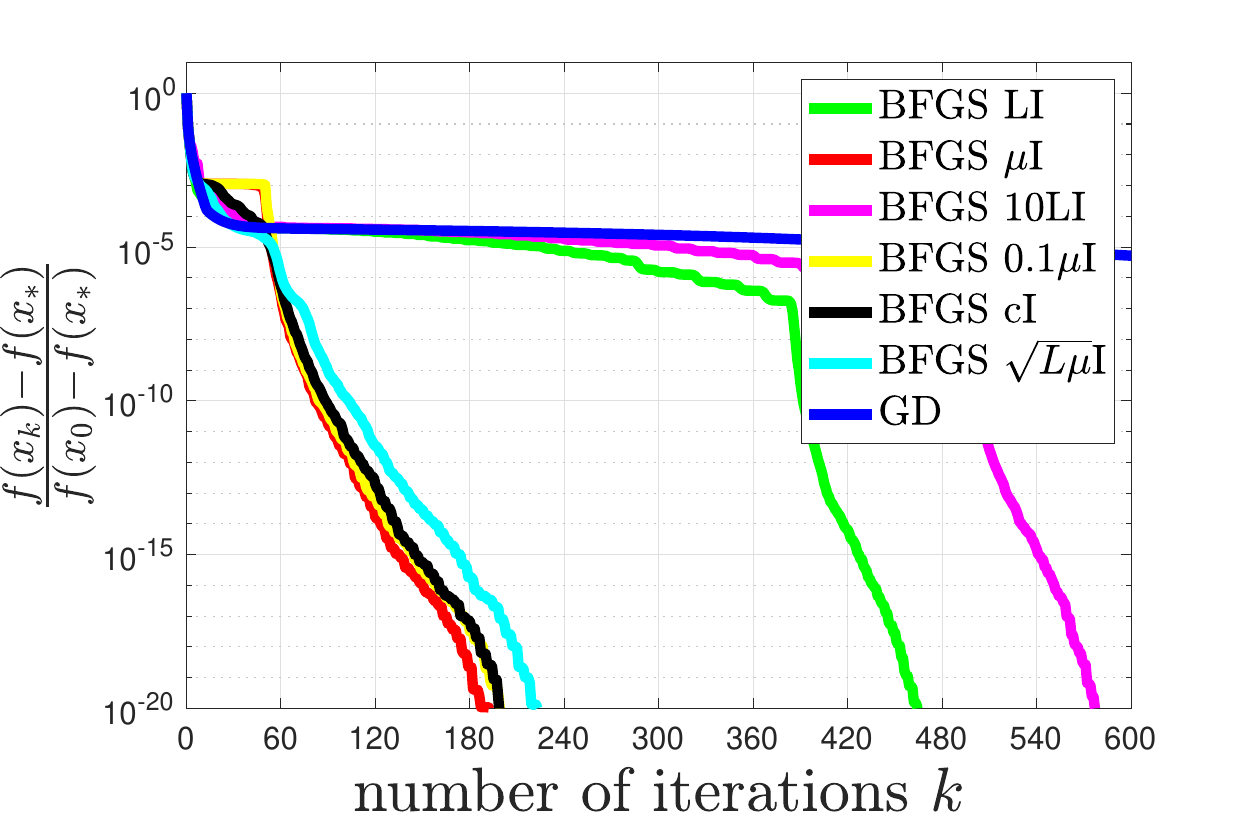}}
    \subfigure[$d=500$, $\kappa=10^3$, $L=1$, $\mu=10^{-3}$, $c=10^{-1}$.]{\includegraphics[width=0.32\linewidth]{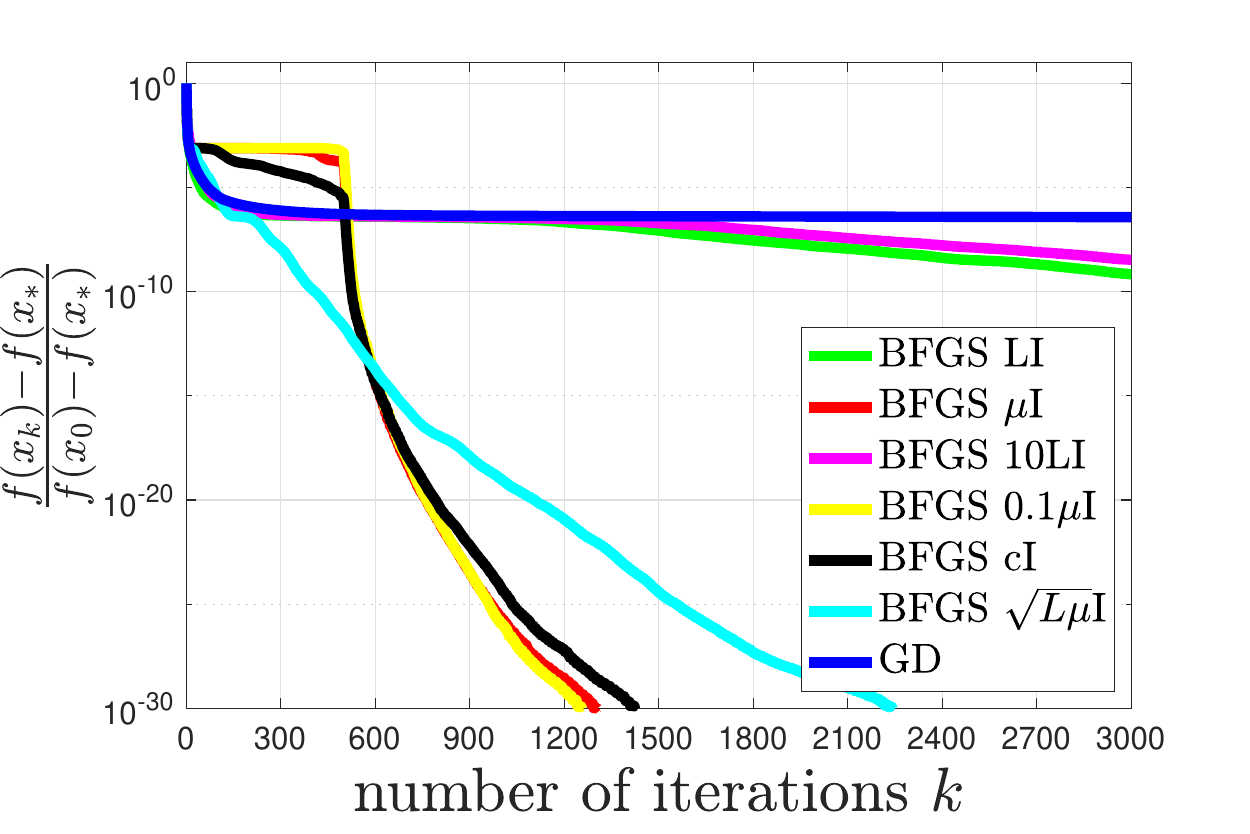}}
    \subfigure[$d=2000$, $\kappa=10^3$, $L=1$, $\mu=10^{-3}$, $c=10^{-1}$.]{\includegraphics[width=0.32\linewidth]{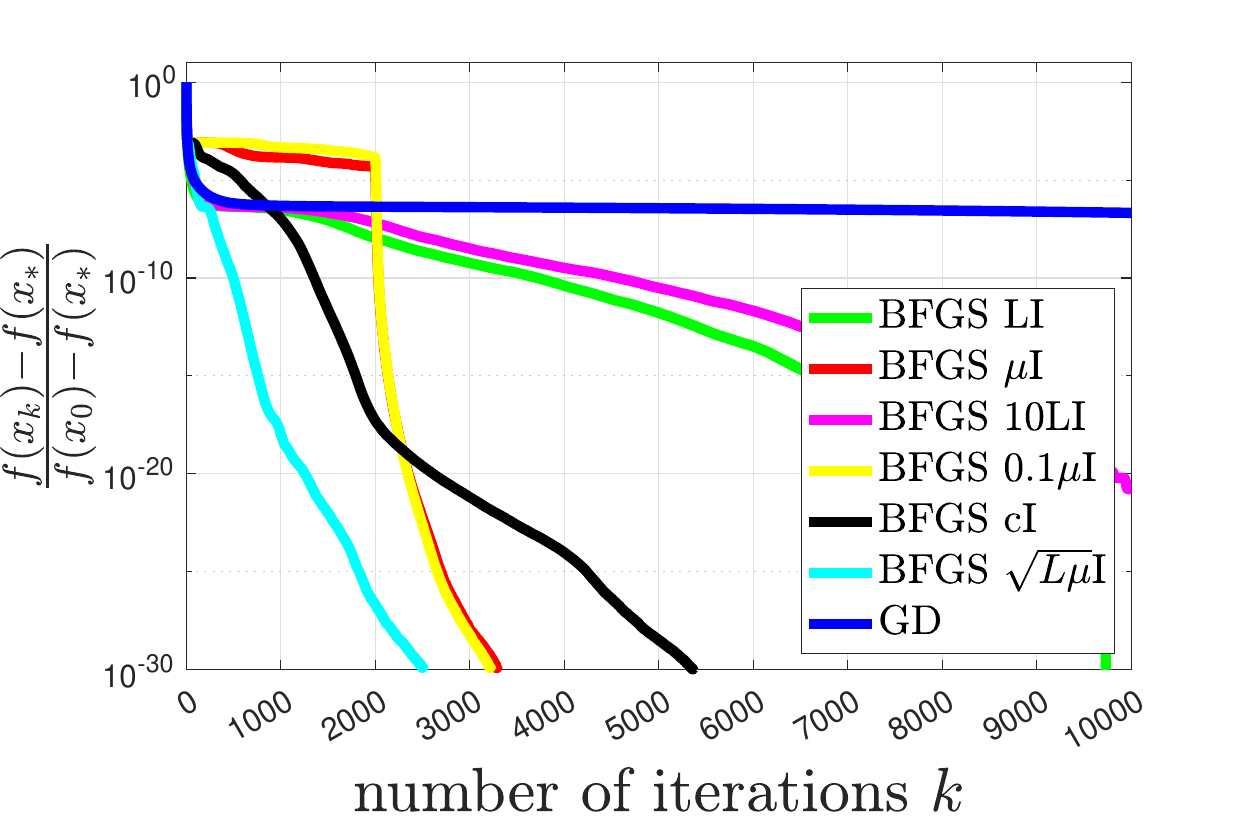}}
    \subfigure[$d=600$, $\kappa=10^2$, $L=1$, $\mu=10^{-2}$, $c=1$.]{\includegraphics[width=0.32\linewidth]{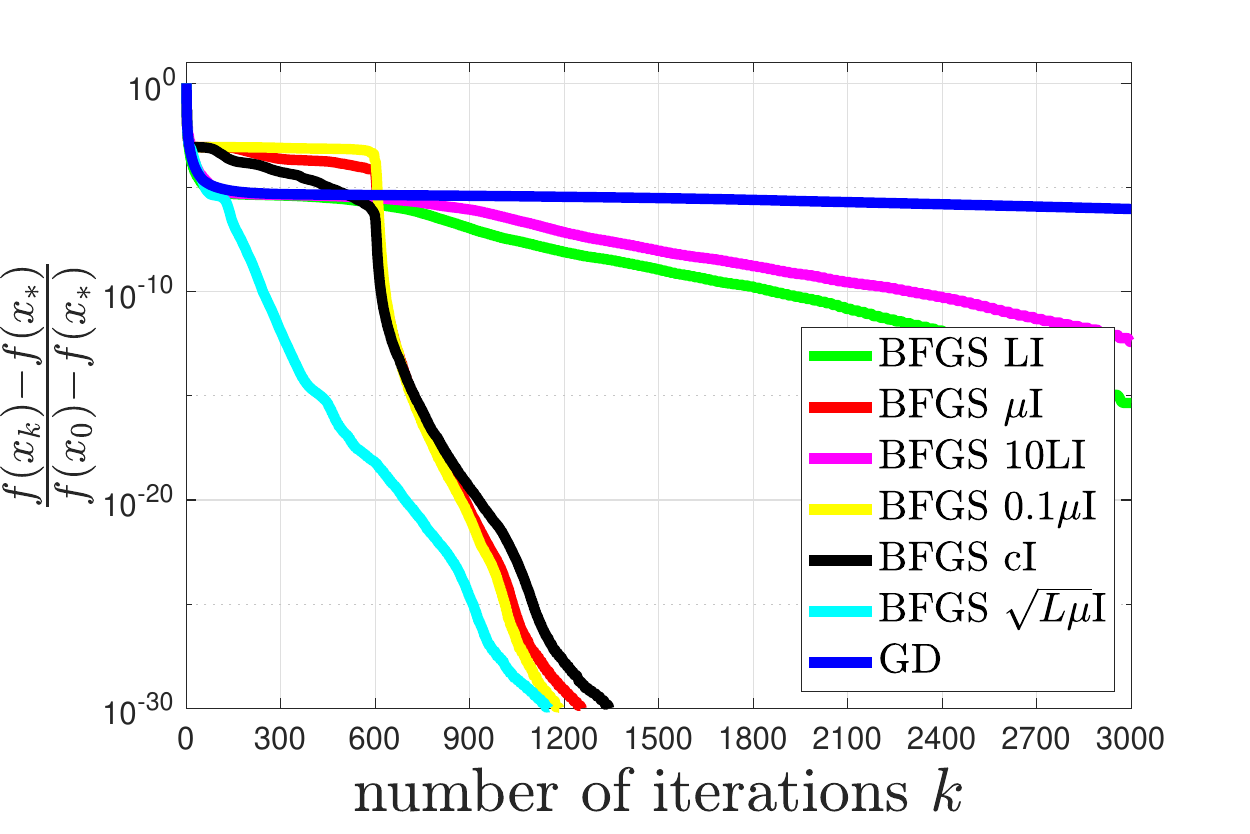}}
    \subfigure[$d=600$, $\kappa=10^3$, $L=1$, $\mu=10^{-3}$, $c=10^{-1}$.]{\includegraphics[width=0.32\linewidth]{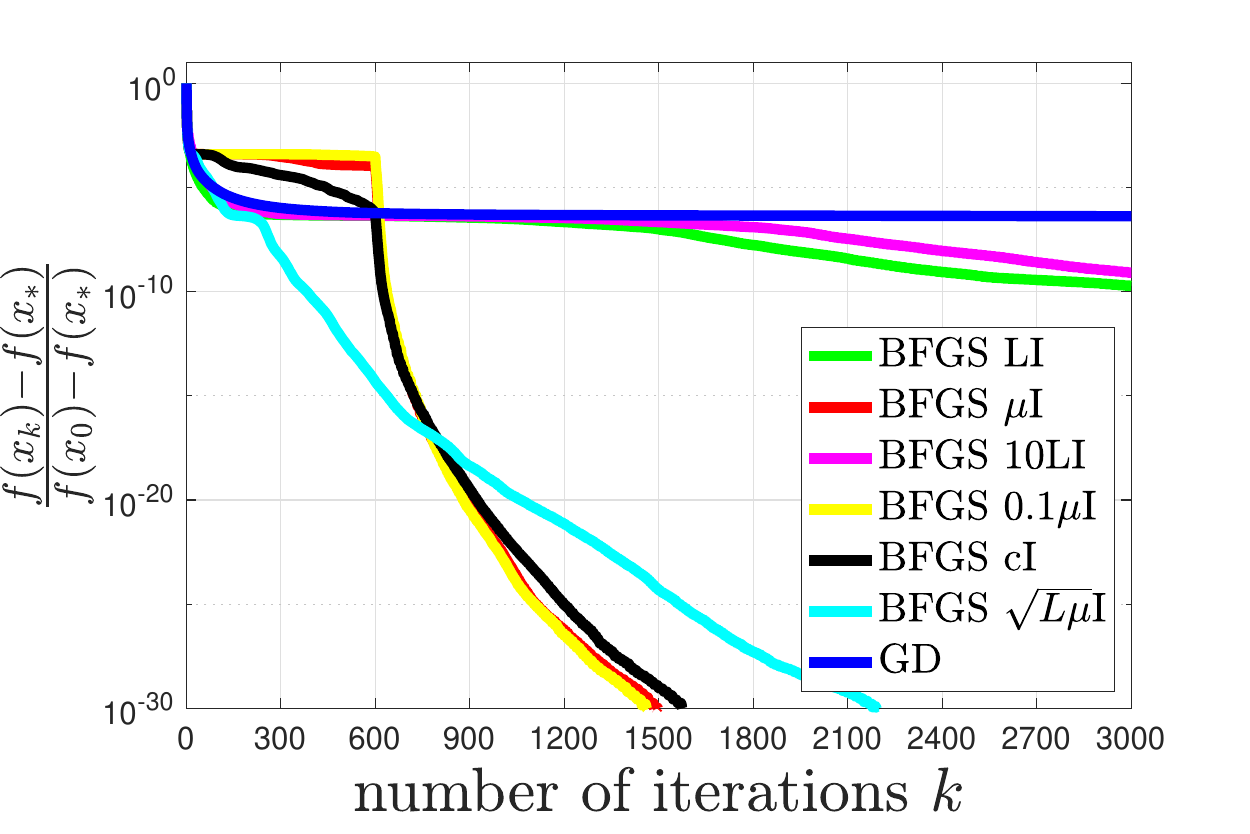}}
    \subfigure[$d=600$, $\kappa=10^4$, $L=1$, $\mu=10^{-4}$, $c=10^{-1}$.]{\includegraphics[width=0.32\linewidth]{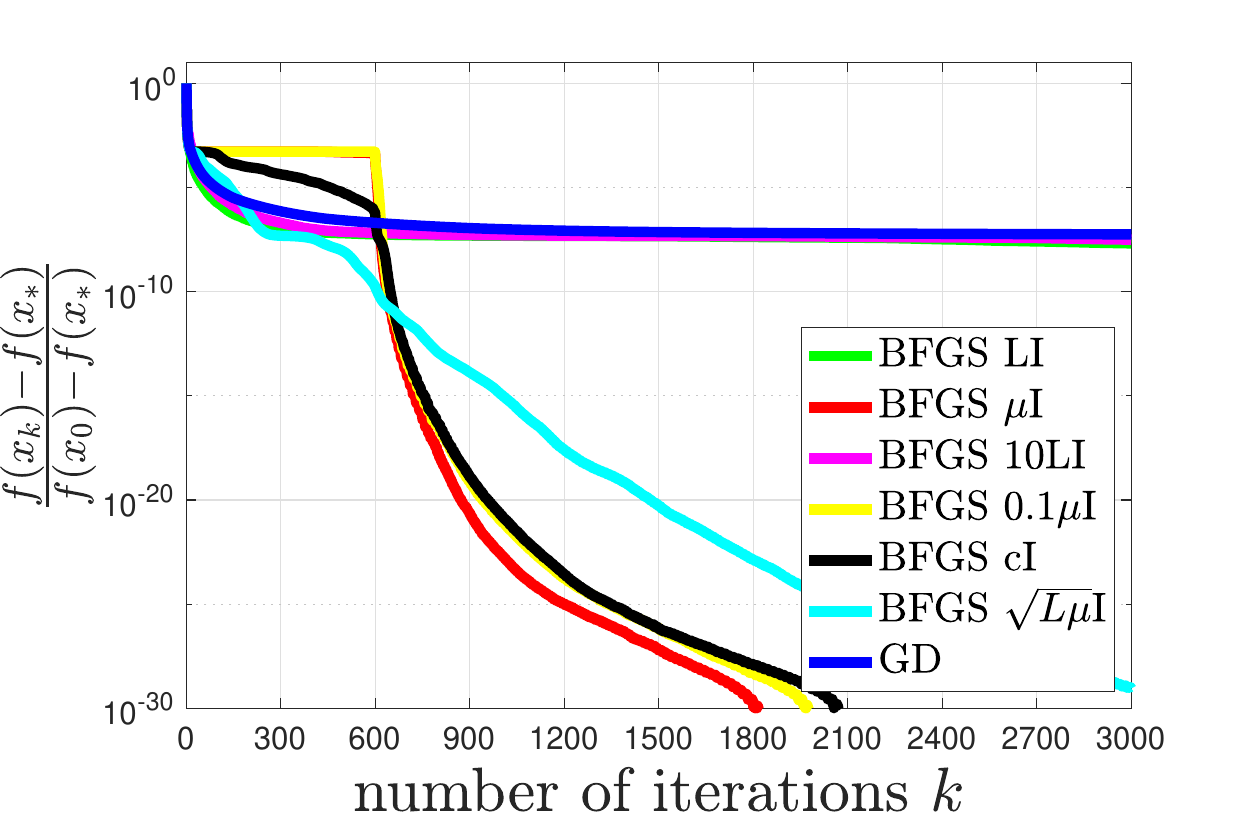}}
    \caption{Convergence rates of BFGS with different $B_0$ and gradient descent for solving the hard cubic objective function when condition number and dimension is varied.}\label{fig:1}
\end{figure}

Several observations are in order. 
\begin{itemize}
    \item BFGS with $B_0 = L I$ initially converges faster than BFGS with $B_0 = \mu I$ in most plots, aligning with our theoretical findings that the linear convergence rate of BFGS with $B_0 = L I$ surpasses that of $B_0 = \mu I$.
    \item The transition to superlinear convergence for BFGS with $B_0 = \mu I$ typically occurs around $k\approx d$, as predicted by our theoretical analysis. Interestingly, this transition does not always coincide with the iterates approaching the solution's local neighborhood; in many cases, it occurs for BFGS with $B_0 = \mu I$ even when its error is larger than that of gradient descent. 
    \item Although BFGS with $B_0 = L I$ initially converges faster, its transition to superlinear convergence consistently occurs later than for $B_0 = \mu I$. Notably, for a fixed dimension $d=600$, the transition to superlinear convergence for $B_0 = L I$ occurs increasingly later as the problem condition number rises, an effect not observed for $B_0 = \mu I$. This phenomenon indicates that the superlinear rate for $B_0 = L I$ is more sensitive to the condition number $\kappa$, which corroborates our theory that the number of iterations required for superlinear convergence is $\mathcal{O}(d\kappa)$ for $B_0 = L I$ and is improved to $\mathcal{O}(d\log{\kappa})$ for $B_0 = \mu I$.
    \item We observe that the performance of BFGS with $B_0=10LI$ is slightly worse than with $B_0 =LI$, while the convergence curve of BFGS with $B_0 = 0.1\mu I$ is almost identical to that with $B_0 = \mu I$. 
    Moreover, the convergence behavior of BFGS with $B_0 = \sqrt{L\mu}I$ is generally similar to that with $B_0 = \mu I$, but it may become slower when the number of iterations is large.
    \item 
    Finally, we observe that BFGS with $B_0 = cI$ initially converges slower than the case where $B_0 = L I$, but faster than $B_0 = \mu I$. After approximately $d$ iterations, the convergence rate of BFGS with $B_0 =cI$ surpasses that of $B_0 = L I$, while being slightly slower than the case where $B_0 = \mu I$. This phenomenon is consistent with the fact that $c \in [\mu, L]$, 
indicating that the performance of $B_0 = c I$ should fall between the performance of $B_0 = L I$ and $B_0 = \mu I$.
These findings align with our theoretical analysis of the trade-off between global linear and superlinear convergence rates for different initial Hessian approximation matrices, as discussed in Sections~\ref{sec:linear} and \ref{sec:superlinear}.
\end{itemize}


\begin{figure}
    \centering
    \subfigure[$d:50$, $\kappa:10^3$, $c:10^{-1}$.]{\includegraphics[width=0.32\linewidth]{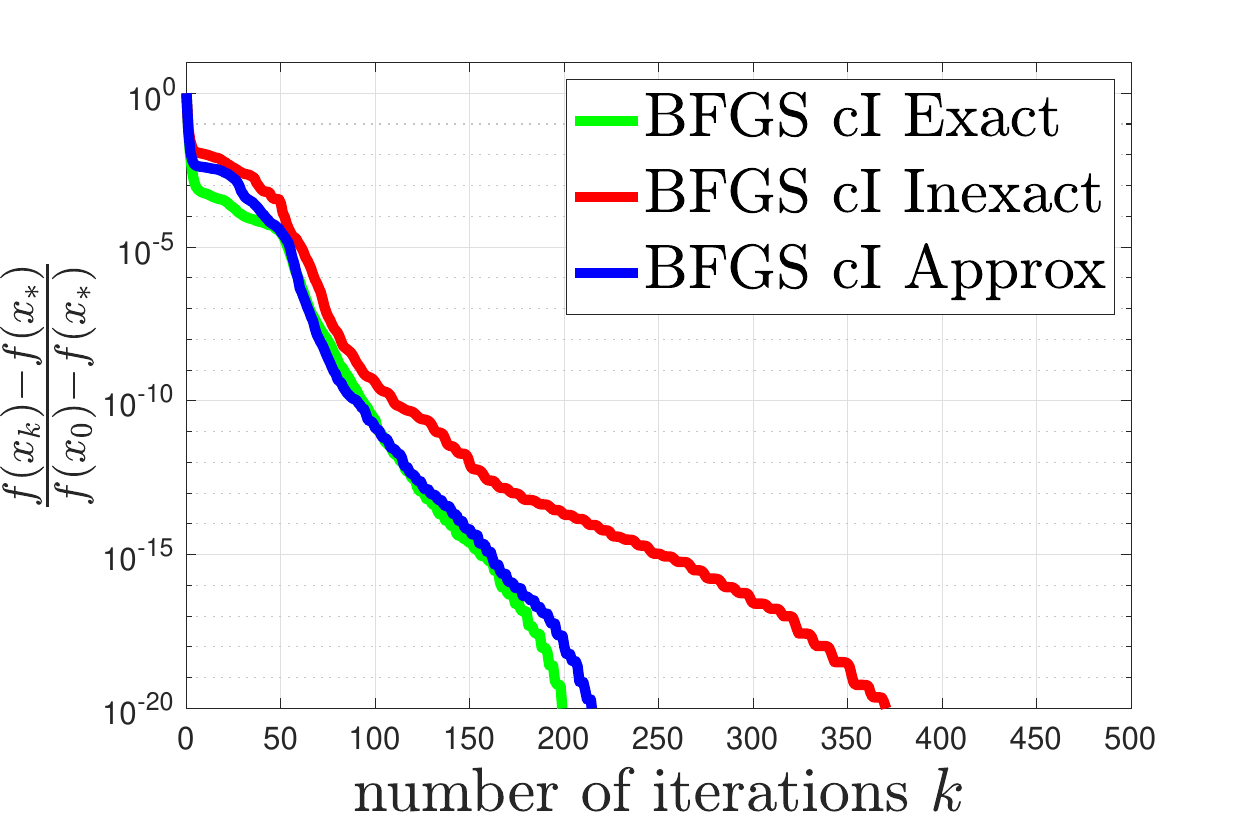}}
    \subfigure[$d:500$, $\kappa:10^3$, $c:10^{-1}$.]{\includegraphics[width=0.32\linewidth]{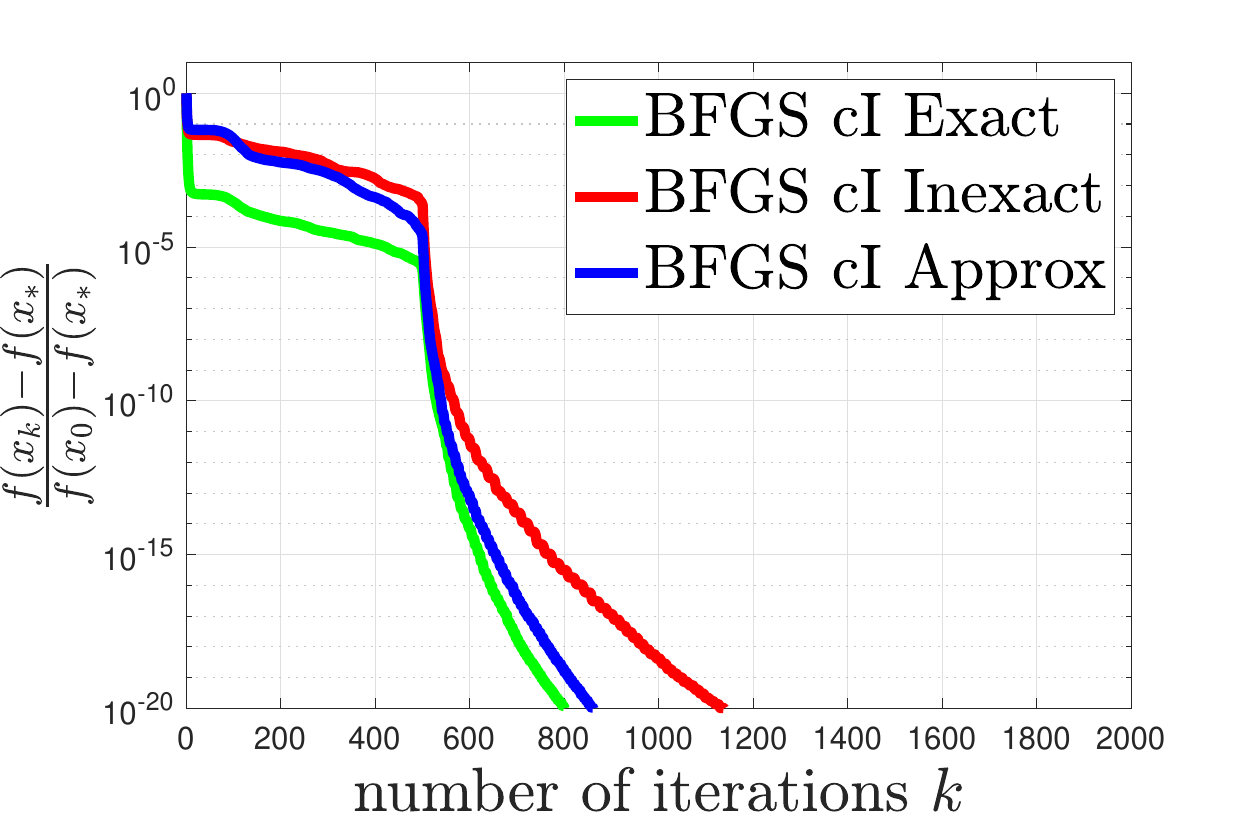}}
    \subfigure[$d:2000$, $\kappa:10^3$, $c:10^{-1}$.]{\includegraphics[width=0.32\linewidth]{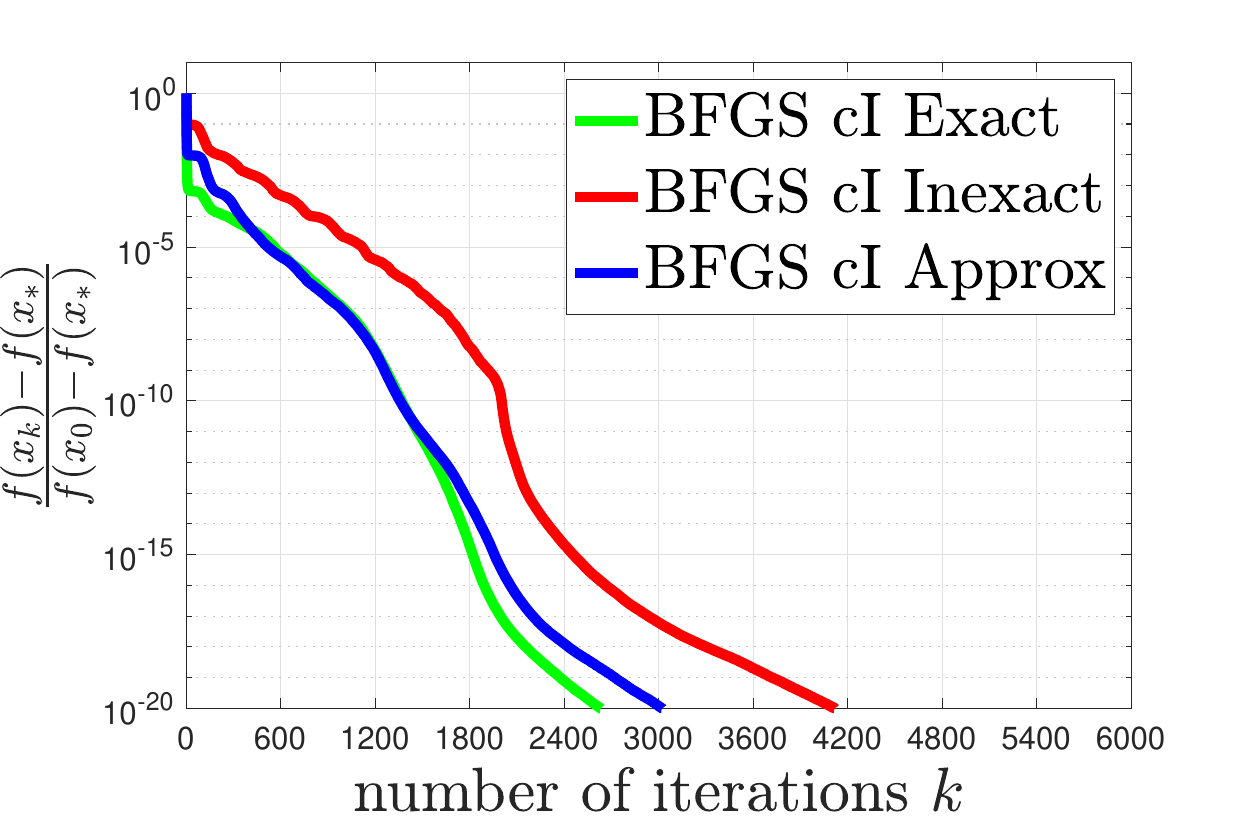}}
    \subfigure[$d:600$, $\kappa:10^2$, $c:1$.]
    {\includegraphics[width=0.32\linewidth]{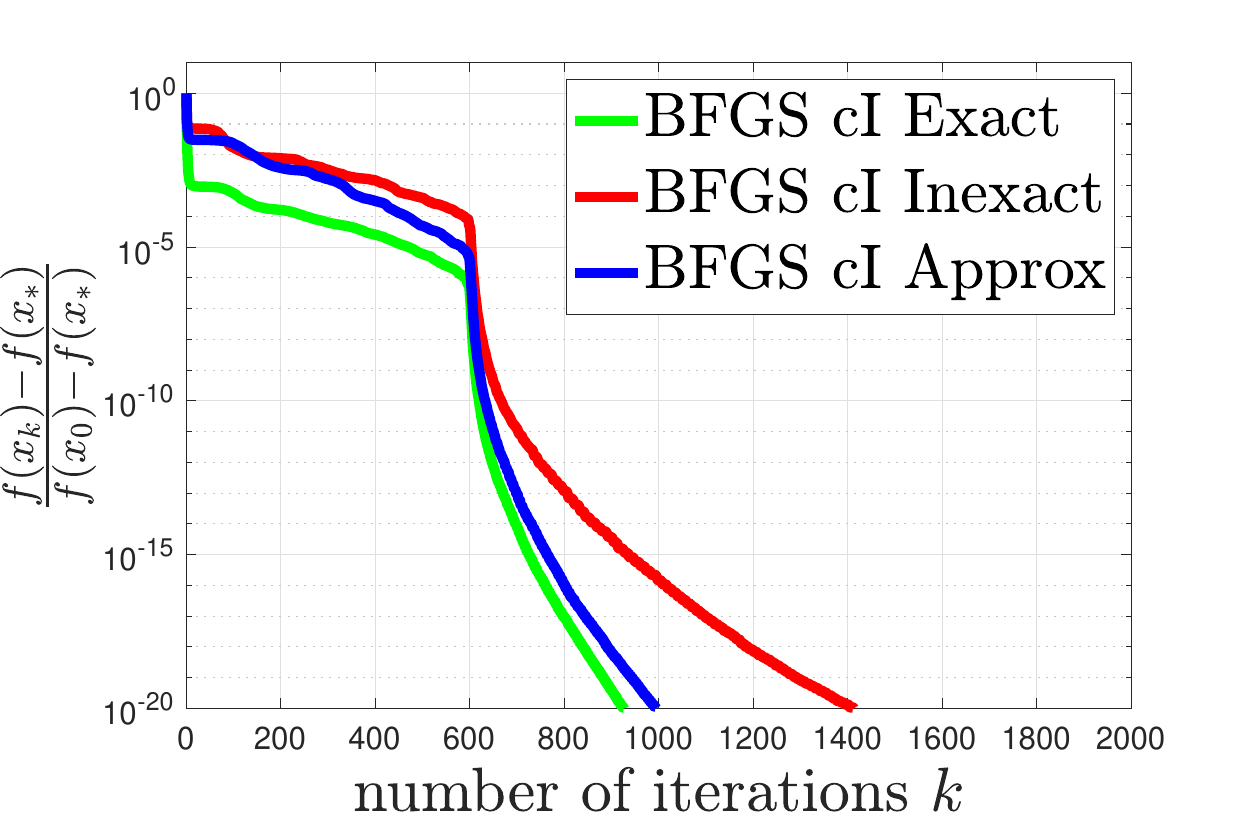}}
    \subfigure[$d:600$, $\kappa:10^3$, $c:10^{-1}$.]{\includegraphics[width=0.32\linewidth]{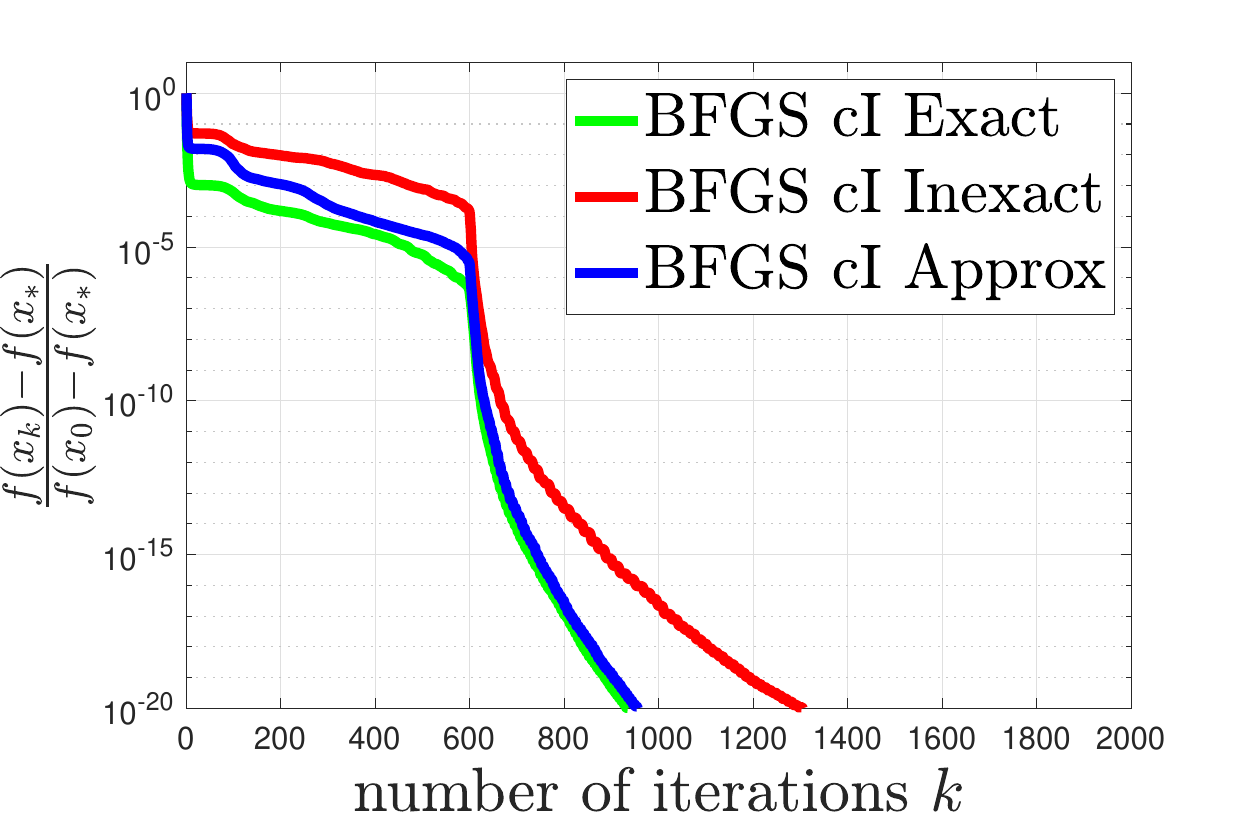}}
    \subfigure[$d:600$, $\kappa:10^4$, $c:10^{-1}$.]{\includegraphics[width=0.32\linewidth]{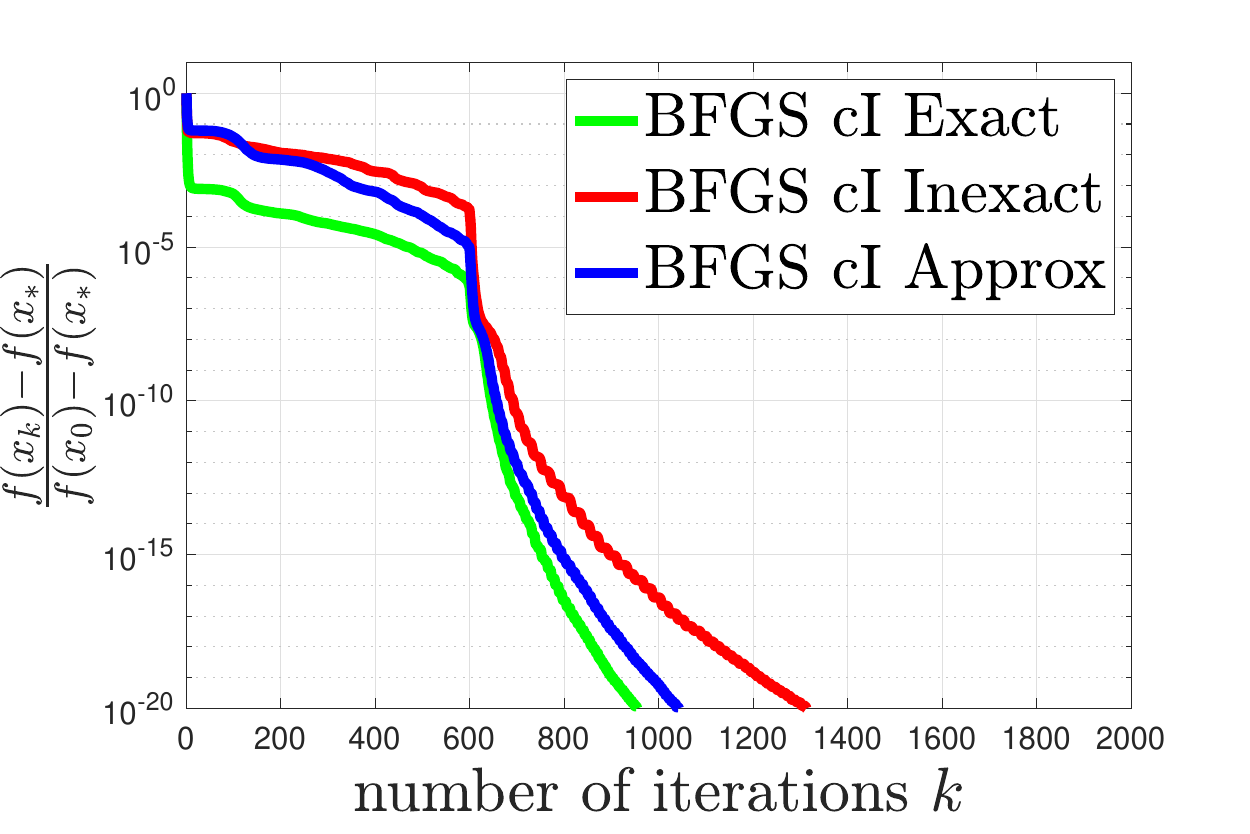}}
    \caption{Convergence rates of BFGS with $B_0 = cI$ for solving the hard cubic objective function with different line search scheme: exact line search, inexact line search and approximate exact line search.}\label{fig:2}
\end{figure}

Additionally, to analyze the sensitivity of BFGS to different line-search schemes, we compare its performance when \( B_0 = cI \) under three distinct line-search strategies, as shown in Figure~\ref{fig:2}. 

The first approach is the \textit{Exact Line Search}, which is the primary focus of our paper. It is implemented using MATLAB's ``fminsearch'' function.  

The second approach is the \textit{Inexact Line Search}, where the step size is determined by enforcing the well-known strong Wolfe conditions:
\begin{align}
    f(x_t + \eta_t d_t) & \leq f(x_t) + \alpha \eta_t\nabla f(x_t)^\top d_t, \label{sufficient_decrease}\\
    |\nabla f(x_t + \eta_t d_t)^\top d_t| & \leq \beta |\nabla f(x_t)^\top d_t|. \label{curvature_condition}
\end{align}
Here, \(\alpha\) and \(\beta\) are line-search parameters that satisfy \( 0 < \alpha < \beta < 1 \) and \( 0 < \alpha < \frac{1}{2} \). To implement this inexact line search, we use the Moré-Thuente line search scheme\footnote{The MATLAB implementation used is available at \url{https://www.cs.umd.edu/users/oleary/software/}.}, which selects a step size \( \eta_t \) at iteration \( t \) that satisfies the strong Wolfe conditions \eqref{sufficient_decrease} and \eqref{curvature_condition}. In our experiments, we set $\alpha = 0.05$ and $\beta = 0.06$ for the above inexact line search and it requires around 10 iterations on average to find $\eta$ satisfying the conditions in \eqref{sufficient_decrease} and \eqref{curvature_condition}. 

The third and final line-search scheme we consider is the \textit{Approximated Exact Line Search}, in which we approximate the solution of the exact line search up to an accuracy of $\epsilon$. Please see Algorithm~\ref{algo} in Appendix~\ref{approx_line_search} for details. 

From Figure~\ref{fig:2}, we observe that the convergence of BFGS with an inexact line search is slightly slower compared to BFGS with an exact line search, whereas BFGS with an approximated exact line search exhibits a convergence behavior nearly identical to the latter.

\section{Conclusion}\label{sec:conclusion}

In this paper, we established explicit global linear and superlinear convergence rates for the BFGS quasi-Newton method with the exact line search scheme, assuming the objective function is strongly convex with a Lipschitz continuous gradient and Hessian. Our results hold for any initial point $x_0 \in \mathbb{R}^{d}$ and any initial Hessian approximation matrix $B_0 \in \mathbb{S}_{++}^{d}$, and they depend on the condition number $\kappa$, the dimension $d$, and the initial function value optimality gap $C_0$. We highlighted the critical role of the initial Hessian approximation matrix in influencing the transition between our established non-asymptotic global linear and superlinear bounds. Furthermore, we specialized our convergence guarantees for different choices of the initial Hessian approximation matrix. Finally, we compared the convergence curves of BFGS with various initial Hessian approximation matrices and line search schemes in the numerical experiments, and the empirical results are consistent with our theoretical analysis.

\section*{Conflict of interest}
The authors declare that they have no conflict of interest.

\appendix

\section*{Appendix}



\section{Proof of Proposition~\texorpdfstring{\ref{lemma_BFGS}}{2}}\label{appen:lemma_BFGS}

First, we show that 
\begin{align}
    \mathbf{Tr}(\hat{B}_{k+ 1}) & = \mathbf{Tr}(\hat{B}_{k}) - \frac{\|\hat{B}_k \hat{s}_k\|^2}{\hat{s}_k^\top \hat{B}_k \hat{s}_k} + \frac{\|\hat{y}_k\|^2}{\hat{s}_k^\top \hat{y}_k}, \label{eq:trace} \\
    \mathbf{Det}(\hat{B}_{k + 1}) & = \mathbf{Det}(\hat{B}_k)\frac{\hat{s}_k^\top \hat{y}_k}{\hat{s}_k^\top \hat{B}_k \hat{s}_k}. \label{eq:determinant}
\end{align}
Taking the trace on both sides of the equation in \eqref{BFGS_weighted} and using the fact that $\mathbf{Tr}(ab^\top) = a^\top b$ for any vectors $a$ and $b$, we obtain the equality in \eqref{eq:trace}. For the proof of \eqref{eq:determinant}, we refer the reader to  \cite[Lemma 6.2]{rodomanov2020rates} . Take the logarithm on both sides of the above equation, we obtain that
\begin{equation*}
    \log{\frac{\hat{s}_k^\top \hat{y}_k}{\hat{s}_k^\top \hat{B}_k \hat{s}_k}} = \log{\mathbf{Det}(\hat{B}_{k + 1})} - \log{\mathbf{Det}(\hat{B}_k)}.
\end{equation*}
Recall that $\hat{m}_k = \frac{\hat{y}_k^\top \hat{s}_k}{\|\hat{s}_k\|^2}$ and $\cos(\hat{\theta}_k) = -\hat{g}_k^\top \hat{s}_k/(\|\hat{g}_k\|\|\hat{s}_k\|)$. Since $\hat{B}_k \hat{s}_k = - \eta_k \hat{g}_k$, we also have $\cos(\hat{\theta}_k) = \hat{s}_k^\top \hat{B}_k \hat{s}_k/ (\|\hat{B}_k \hat{s}_k\|\|\hat{s}_k\|)$. Hence, we can write 
\begin{equation*}
    \frac{\hat{s}_k^\top \hat{y}_k}{\hat{s}_k^\top \hat{B}_k \hat{s}_k} = \frac{\|\hat{B}_k \hat{s}_k\|^2\|\hat{s}_k\|^2}{(\hat{s}_k^\top \hat{B}_k \hat{s}_k)^2} \frac{\hat{s}_k^\top \hat{y}_k}{\|\hat{s}_k\|^2} \frac{\hat{s}_k^\top \hat{B}_k \hat{s}_k}{\|\hat{B}_k \hat{s}_k\|^2} = \frac{\hat{m}_k}{\cos^2(\hat{\theta}_k)} \frac{\hat{s}_k^\top \hat{B}_k \hat{s}_k}{\|\hat{B}_k \hat{s}_k\|^2}. 
\end{equation*}
Thus, we obtain that 
\begin{align*}
    \Psi(\hat{B}_{k+1}) - \Psi(\hat{B}_k) & = \mathbf{Tr}(\hat{B}_{k+ 1}) - \mathbf{Tr}(\hat{B}_{k}) + \log\mathbf{Det}(\hat{B}_k) - \log{\mathbf{Det}(\hat{B}_{k + 1})} \\
    & = \frac{\|\hat{y}_k\|^2}{\hat{s}_k^\top \hat{y}_k} - \frac{\|\hat{B}_k \hat{s}_k\|^2}{\hat{s}_k^\top \hat{B}_k \hat{s}_k} - \log{\frac{\hat{s}_k^\top \hat{y}_k}{\hat{s}_k^\top \hat{B}_k \hat{s}_k}} \\
    & = \frac{\|\hat{y}_k\|^2}{\hat{s}_k^\top \hat{y}_k} - 1 + \log \frac{\cos^2\hat{\theta}_k}{\hat{m}_k} - \left(\frac{\|\hat{B}_k \hat{s}_k\|^2}{\hat{s}_k^\top \hat{B}_k \hat{s}_k}- \log \frac{\|\hat{B}_k \hat{s}_k\|^2}{\hat{s}_k^\top \hat{B}_k \hat{s}_k}+1\right) \\
    & \leq \frac{\|\hat{y}_k\|^2}{\hat{s}_k^\top \hat{y}_k} - 1 + \log \frac{\cos^2\hat{\theta}_k}{\hat{m}_k}.
\end{align*}
where the last inequality holds since $x - \log x + 1 \geq 0$ for any $x > 0$. Hence, \eqref{eq:potential_decrease} follows from the above inequality. Finally, the result in \eqref{eq:sum_of_logs} follows from summing both sides of \eqref{eq:potential_decrease} from $i = 0$ to $k - 1$, i.e.,
\begin{align*}
    \Psi(\hat{B}_{k}) \leq \Psi(\hat{B}_{0}) + \sum_{i = 0}^{k - 1}\left(\frac{\|\hat{y}_i\|^2}{\hat{s}_i^\top \hat{y}_i} - 1\right) + \sum_{i = 0}^{k - 1} \log \frac{\cos^2\hat{\theta}_i}{\hat{m}_i},
\end{align*}
which further implies that 
\begin{align*}
    \sum_{i = 0}^{k - 1} \log{\frac{\cos^2(\hat{\theta}_i)}{\hat{m}_i}} \geq  \Psi(\hat{B}_{k}) - \Psi(\hat{B}_{0}) + \sum_{i = 0}^{k - 1}\left(1-\frac{\|\hat{y}_i\|^2}{\hat{s}_i^\top \hat{y}_i} \right) \geq - \Psi(\hat{B}_{0}) + \sum_{i = 0}^{k - 1}\left(1-\frac{\|\hat{y}_i\|^2}{\hat{s}_i^\top \hat{y}_i} \right),
\end{align*}
where the last inequality holds since $\Psi(\hat{B}_{k}) \geq 0$ for any $k \geq 0$.

\section{Proof of Lemma~\texorpdfstring{\ref{lemma_Hessian}}{2}}\label{proof_of_lemma_Hessian}

\begin{enumerate}[(a)]
    
    \item Recall that $J_k = \int_{0}^{1}\nabla^2{f(x_k + \tau (x_{k + 1} - x_k))}d\tau$. Using the triangle inequality, we have 
    \begin{align*}
        \|\nabla^2{f(x_{*})} - J_k\| &= \left\| \int_{0}^{1}\!\!\left(\nabla^2{f(x_{*})} - \nabla^2{f(x_k + \tau (x_{k+1}-x_k))} \right)d\tau\right\| \\
        &\leq \int_{0}^{1}\|\nabla^2{f(x_{*})} - \nabla^2{f(x_k + \tau (x_{k+1}-x_k))} \|d\tau. 
    \end{align*}
    Moreover, it follows from Assumption~\ref{ass_Hess_lip} that $ \|\nabla^2{f(x_{*})} - \nabla^2{f(x_k + \tau (x_{k+1}-x_k))} \| \leq M \|(1 - \tau)(x_* - x_k) + \tau (x_* - x_{k + 1})\|$ for any $\tau \in [0,1]$. Thus, we can further apply the triangle inequality to obtain  
    \begin{align*}
        \|\nabla^2{f(x_{*})} - J_k\| & \leq \int_{0}^{1}M \|(1 - \tau)(x_* - x_k) + \tau (x_* - x_{k + 1})\|d\tau \\
        & \leq M  \|x_k - x_*\|\int_{0}^{1}(1 - \tau)d\tau + M \|x_{k + 1} - x_*\|\int_{0}^{1} \tau d\tau \\
        & =  \frac{M}{2}(\|x_k - x_*\| + \|x_{k + 1} - x_*\|). 
    \end{align*}
    Since $f$ is strongly convex, by Assumption~\ref{ass_str_cvx} and $f(x_{k+1}) \leq f(x_k)$, we have $\frac{\mu}{2}\|x_k - x_*\|^2 \leq f(x_k) - f(x_*)$, which implies that $\|x_k-x_*\| \leq \sqrt{2(f(x_k)-f(x_*))/\mu}$. Similarly, since $f(x_{k+1}) \leq f(x_k)$, it also holds that  $ \|x_{k+1}-x_*\| \leq \sqrt{2(f(x_{k+1})-f(x_*))/\mu} \leq \sqrt{2(f(x_{k})-f(x_*))/\mu}$. Hence, we obtain 
    \begin{equation}\label{eq:H_*-J_k}
        \|\nabla^2{f(x_{*})} - J_k\| \leq \frac{M}{\sqrt{\mu}}\sqrt{2(f(x_k) - f(x_*))}
    \end{equation}
    Moreover, notice that by Assumption~\ref{ass_str_cvx}, we also have $J_k \succeq \mu I$ and $\nabla^2 f(x_*) \succeq \mu I$. Hence, \eqref{eq:H_*-J_k} implies that 
    \begin{align*}
        \nabla^2{f(x_{*})} - J_k &\preceq \|\nabla^2{f(x_{*})} - J_k\|I \preceq \frac{M}{\mu^{\frac{3}{2}}}\sqrt{2(f(x_k) - f(x_*))}J_k \leq C_k J_k, \\
        J_k - \nabla^2{f(x_{*})} &\preceq \|J_k - \nabla^2{f(x_{*})}\|I \preceq \frac{M}{\mu^{\frac{3}{2}}}\sqrt{2(f(x_k) - f(x_*))}\nabla^2{f(x_{*})} \leq C_k \nabla^2{f(x_{*})}.
    \end{align*}
    where we used the definition of $C_k$ in \eqref{distance}. By rearranging the terms, we obtain \eqref{eq:J_k_vs_H*}.  
    
    \item Recall that $G_k = \int_{0}^{1}\nabla^2{f(x_k + \tau (x_{*} - x_k))}d\tau$. Similar to the arguments in (a), we have 
    \begin{equation}\label{eq:H_*-G_k}
        \begin{split}
            \left\| \nabla^2{f(x_{*})} - G_k \right\|
            & = \left\|\int_{0}^{1}\left(\nabla^2{f(x_{*})} - \nabla^2{f(x_k + \tau (x_{*} - x_k))}\right)d\tau \right\| \\
            & \leq \int_{0}^{1}\|\nabla^2{f(x_{*})} - \nabla^2{f(x_k + \tau (x_{*} - x_k))}\|d\tau \\
            & \leq M\int_{0}^{1}\|(1 - \tau)(x_* - x_k)\| d\tau = M\|x_k - x_*\|\int_{0}^{1}(1 - \tau)d\tau \\
            & = \frac{M}{2}\|x_k - x_*\|  \leq \frac{M}{\sqrt{\mu}}\sqrt{2(f(x_k) - f(x_*))}.
        \end{split}
    \end{equation}
    Moreover, notice that by Assumption~\ref{ass_str_cvx} we also have $G_k \succeq \mu I$ and $\nabla^2 f(x_*) \succeq \mu I$.
    The rest follows similarly as in the proof of (a) and we prove \eqref{eq:G_k_vs_H*}.

    \item 
    For any $\hat{\tau} \in [0, 1]$, we have
    \begin{equation*}
        \begin{split}
            & \phantom{{}={}}\left\| \nabla^2{f(x_k + \hat{\tau} (x_{k + 1} - x_{k}))} -  \nabla^2{f(x_*)}\right\| \\
            & \leq M\|x_k + \hat{\tau} (x_{k + 1} - x_{k}) - x_*\| \leq M (\hat{\tau} \|x_{k + 1} - x_*\| + (1 - \hat{\tau}) \|x_{k} - x_*\| ) \\
            & \leq \frac{M}{\sqrt{\mu}}(\hat{\tau}\sqrt{2(f(x_{k + 1}) - f(x_*))} + (1 - \hat{\tau}) \sqrt{2(f(x_{k}) - f(x_*))}) \\
            & \leq \frac{M}{\sqrt{\mu}}(\hat{\tau}\sqrt{2(f(x_{k}) - f(x_*))} + (1 - \hat{\tau}) \sqrt{2(f(x_{k}) - f(x_*))}) \\
            & = \frac{M}{\sqrt{\mu}}\sqrt{2(f(x_{k}) - f(x_*))}.
        \end{split}
    \end{equation*}
    Together with \eqref{eq:H_*-J_k}, it follows from the triangle inequality that 
    \begin{equation*}
        \begin{split}
            & \phantom{{}={}}\left\| \nabla^2{f(x_k + \hat{\tau} (x_{k + 1} - x_{k}))} -  J_k\right\| \\
            & \leq \|\nabla^2{f(x_k + \hat{\tau} (x_{k + 1} - x_{k}))} -  \nabla^2{f(x_*)}\| + \|\nabla^2{f(x_*)} -  J_k\| \\
            & \leq \frac{2M}{\sqrt{\mu}}\sqrt{2(f(x_k) - f(x_*))}.
        \end{split}
    \end{equation*}
    Moreover, notice that by Assumption~\ref{ass_str_cvx}, we also have $\nabla^2{f(x_k + \hat{\tau} (x_{k + 1} - x_{k}))} \succeq \mu I$ and $J_k \succeq \mu I$. The rest follows similarly as in the proof of (a) and we prove \eqref{eq:middle_hessian_vs_J_k}.
    
    \item 
    For any $\tilde{\tau} \in [0, 1]$, we have that
    \begin{equation*}
        \begin{split}
            \left\| \nabla^2{f(x_k + \tilde{\tau} (x_{*} - x_{k}))} -  \nabla^2{f(x_{*})} \right\|
            & \leq M\|x_k + \tilde{\tau} (x_{*} - x_{k})) - x_*\| \\
            & = M(1 - \tilde{\tau})\|x_{*} - x_{k}\| \\
            & \leq M \|x_{*} - x_{k}\| \\
            & \leq \frac{M}{\sqrt{\mu}}\sqrt{2(f(x_k) - f(x_*))}.
        \end{split}
    \end{equation*}
    Together with~\eqref{eq:H_*-G_k}, it follows form the triangle inequality that 
    \begin{equation*}
        \begin{split}
        & \phantom{{}={}}\left\| \nabla^2{f(x_k + \tilde{\tau} (x_{*} - x_{k}))} -  G_k\right\| \\
        & \leq \left\| \nabla^2{f(x_k + \tilde{\tau} (x_{*} - x_{k}))} -  \nabla^2{f(x_{*})}\right\| + \left\| \nabla^2{f(x_{*})} -  G_k\right\| \\
        & \leq \frac{2M}{\sqrt{\mu}}\sqrt{2(f(x_k) - f(x_*))}.
    \end{split}
    \end{equation*}
    Moreover, notice that by Assumption~\ref{ass_str_cvx}, we also have $\nabla^2{f(x_k + \tilde{\tau} (x_{*} - x_{k}))} \succeq \mu I$ and $G_k \succeq \mu I$. The rest follows similarly as in the proof of (a) and we prove \eqref{eq:middle_hessian_vs_G_k}.

\end{enumerate}


\section{Approximate Exact Line Search Algorithm}\label{approx_line_search}

The approximation exact line search is implemented using the bisection Algorithm~\ref{algo} with $\epsilon$ as the approximation error. The key idea is to select $\eta$ such that $\nabla f(x_t + \eta d_t)^\top d_t \approx 0$, since the exact line search step size $\eta_{\text{exact}}$ satisfies the condition $\nabla f(x_t + \eta_{\text{exact}} d_t)^\top d_t = 0$. We begin with an initial step size of $\eta = 1$ and iteratively double it until $\nabla f(x_t + \eta d_t)^\top d_t > 0$. Once this condition is met, we apply the bisection algorithm, leveraging the sign of $\nabla f(x_t + \eta d_t)^\top d_t$ to refine $\eta$. The bisection algorithm is well-suited for this task because the function \( h(\eta) = \nabla f(x_t + \eta d_t)^\top d_t \) is strictly increasing. This follows from the strong convexity of the objective function \( f \), which ensures that \( h'(\eta) = d_t^\top \nabla^2 f(x_t + \eta d_t) d_t > 0 \). Additionally, we note that \( h(0) = \nabla f(x_t)^\top d_t = -g_t^\top B_t^{-1}g_t < 0 \), since \( B_t \) is symmetric positive definite, and that \( h(\eta_{\text{exact}}) = 0 \) with \( \eta_{\text{exact}} > 0 \). In our experiments, we set the required accuracy for this scheme to be $\epsilon = 10^{-8}$ and we observe that on average after 15 iterations the bisection method converges. 

\begin{algorithm}[ht!]
\caption{Bisection Algorithm for Approximation Exact Line Search}\label{algo} 
\begin{algorithmic}
\State \textbf{Input:} Initialized step size $\eta = 1$ and approximation error $\epsilon$
\While {$\nabla{f}(x_t + \eta d_t)^\top d_t < 0$}
    \State $\eta = 2\eta$
\EndWhile
\State Set $\eta_{\min} = 0$ and $\eta_{\max} = \eta$
\While {$\eta_{\max} - \eta_{\min} > \epsilon$}
    \State $\eta = (\eta_{\max} + \eta_{\min})/2$
    \If{$\nabla{f}(x_t + \eta d_t)^\top d_t > 0$}
        \State $\eta_{\max} = \eta$
    \ElsIf{$\nabla{f}(x_t + \eta d_t)^\top d_t < 0$}
        \State $\eta_{\min} = \eta$
    \Else
        \State \textbf{break}
    \EndIf
\EndWhile
\end{algorithmic}
\end{algorithm}

\newpage

\printbibliography

\end{document}